\documentclass[fontsize=12pt,a4paper,headings=normal,
twoside=false,leqno,parskip=half-,abstract=true]{scrartcl}
\usepackage[english]{babel}
\usepackage{booktabs}

\usepackage[utf8]{inputenc}
\usepackage{hyperref}
\hypersetup{
 pdftitle={},
 pdfauthor={},
 colorlinks=true,
 linkcolor=blue,
 citecolor=blue,
 filecolor=blue,
 urlcolor=blue}

\usepackage[pagewise]{lineno}%\linenumbers
\usepackage[version=4]{mhchem}
\usepackage{mathtools} 
\usepackage[format=plain,labelfont=bf,font=small]{caption}
\usepackage{subfigure}%replaced by subcaption
\usepackage{xcolor}
\usepackage[arrow, matrix, curve]{xy}
\usepackage{tikz}
\usepackage{float}
\usepackage{orcidlink}

\usepackage{academicons}
\definecolor{orcidlogocol}{HTML}{A6CE39}
\usepackage{tikz-cd}
\usepackage{caption}
\usepackage{amsmath,amsthm}
\usepackage{amssymb} 
\usepackage[normalem]{ulem}
\usepackage{enumitem}
\usepackage[normalem]{ulem}
\usepackage[version=4]{mhchem}

\newtheorem{theorem}{Theorem}[section]
\newtheorem{conjecture}{Conjecture}[section]
\newtheorem{lemma}[theorem]{Lemma}
\newtheorem{cor}[theorem]{Corollary}
\newtheorem{prop}[theorem]{Proposition}

\theoremstyle{definition}

\theoremstyle{remark}
\newtheorem{remark}[theorem]{Remark}

\numberwithin{equation}{section}

\DeclareMathAlphabet{\mathpzc}{OT1}{pzc}{m}{it}

\newcommand{\be}{\begin{equation}}
\newcommand{\ee}{\end{equation}}
\newcommand{\bea}{\begin{eqnarray}}
\newcommand{\eea}{\end{eqnarray}}

\title{\Large Sequential and distributive dual futile cycle:\\ 
Hopf bifurcation can occur under parameter-rich kinetics \\ 
but cannot occur under mass action kinetics}
\author{\large Nicola Vassena \orcidlink{0000-0001-5411-4976} \thanks
{Universität Leipzig, Germany, \texttt{nicola.vassena@uni-leipzig.de}}
 }
\date{\large \today}

\graphicspath{{./images/} }

\begin{document}

\maketitle

\vspace{-0.5cm}

\begin{abstract}
This paper establishes that the system of ordinary differential equations arising from the sequential and distributive dual futile cycle has the structural capacity for Hopf bifurcations, whenever it is endowed with general parameter-rich kinetics, but it loses such capacity if it is endowed with mass action kinetics.  The proof of the latter fact relies on a Routh--Hurwitz approach, improved by few preliminary structural considerations, but a decisive contribution came from \texttt{GPT-5.6 Sol}, which provided a nontrivial positivity certificate. A second purpose of the paper is indeed to document such AI-contribution, exploiting this well-studied model, and to offer an example of the utilizability of such tools in the context of computer algebra and positivity certificates for large polynomials.
\end{abstract}

\tableofcontents

\section{Introduction}
The sequential and distributive dual futile cycle,
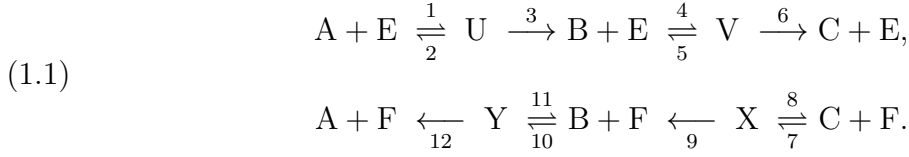
\begin{equation}\label{double}
\begin{alignedat}{3}
\ce{A}+\ce{E}
&\;\overset{1}{\underset{2}{\rightleftharpoons}}\;
\ce{U}\;
\overset{3}{\longrightarrow}\;
&\ce{B}+\ce{E}
&\;\overset{4}{\underset{5}{\rightleftharpoons}}\;
\ce{V}\;
\overset{6}{\longrightarrow}\;
&\ce{C}+\ce{E},
\\[2mm]
\ce{A}+\ce{F}
&\;\underset{12}{\longleftarrow}\;
\ce{Y}\;
\overset{11}{\underset{10}{\rightleftharpoons}}\;
&\ce{B}+\ce{F}
&\;\underset{9}{\longleftarrow}\;
\ce{X}\;
\overset{8}{\underset{7}{\rightleftharpoons}}\;
&\ce{C}+\ce{F}.
\end{alignedat}
\end{equation}
 is a basic model of multi-site phosphorylation, a mechanism central to cellular signaling \cite{cohen2000regulation, gunawardena:2007, ThomsonGunawardena:09}. Species $\{\ce{A},\ce{B},\ce{C}\}$ represent substrates, species $\{\ce{E},\ce{F}\}$ represent enzymes, and species $\{\ce{U},\ce{V},\ce{X},\ce{Y}\}$ represent intermediates. In addition to its modeling relevance in biochemistry, the mathematical interest in such a deceptively simple model resides in the fact that the system is extremely amenable to mathematical testing in the field of chemical reaction network theory \cite{WangSontag:2008, WangSontagMonotone:2008,AliciaetalToric:2012,HolsteinFlockerziConradi:2013,FlockerziHolsteinConradi:2014,ConradiMincheva:2014,HellRendall:2015,ErramiEtAl:2015,ConradiFeliuMinchevaWiuf:2017,BozemanMorales:2017,Tung:2018,Feliu:2020,TorresFeliu21,conradi2024distributive,CaiHimmelmannOstermann:2025}.

The absence of Hopf bifurcations in the associated mass-action system, or more generally the impossibility for its Jacobian to possess a pair of purely-imaginary eigenvalues $\pm i\omega$ with $\omega>0$, has appeared in the community both on-stage \cite{Dickenstein2019, conradietal19,CarstenHopfExclusion19, Dickenstein20, EliNidhiExc:2024, Conradispatial} and off-stage, with many public and private discussions underscoring the interest to settle this technical fact. This is the content of this paper.

To introduce the problem in more detail, let us first assume that each reaction rate is a positive function that only depends on the concentrations of its reactants and has strictly positive partial derivatives throughout the positive orthant.
The reaction network yields the following system of ordinary differential equations for the concentrations:
\begin{align} \label{doubleeq}
\begin{cases}
[\dot{\ce{A}}]=-r_{1}([\ce{A}],[\ce{E}])+r_2([\ce{U}])+r_{12}([\ce{Y}]),\\
[\dot{\ce{B}}]=r_3([\ce{U}])-r_4([\ce{B}],[\ce{E}])+r_5([\ce{V}])+r_9([\ce{X}])-r_{10}([\ce{B}],[\ce{F}])+r_{11}([\ce{Y}]),\\
[\dot{\ce{C}}]=-r_7([\ce{C}],[\ce{F}])+r_6([\ce{V}])+r_8([\ce{X}]),\\
[\dot{\ce{E}}]=-r_{1}([\ce{A}],[\ce{E}])+r_2([\ce{U}])+r_3([\ce{U}])-r_4([\ce{B}],[\ce{E}])+r_5([\ce{V}])+r_6([\ce{V}]),\\
[\dot{\ce{F}}]=-r_7([\ce{C}],[\ce{F}])+r_8([\ce{X}])+r_9([\ce{X}])-r_{10}([\ce{B}],[\ce{F}])+r_{11}([\ce{Y}])+r_{12}([\ce{Y}]),\\
[\dot{\ce{U}}]=r_{1}([\ce{A}],[\ce{E}])-r_2([\ce{U}])-r_3([\ce{U}]),\\
[\dot{\ce{V}}]=r_4([\ce{B}],[\ce{E}])-r_5([\ce{V}])-r_6([\ce{V}]),\\
[\dot{\ce{X}}]=r_7([\ce{C}],[\ce{F}])-r_8([\ce{X}])-r_9([\ce{X}]),\\
[\dot{\ce{Y}}]=r_{10}([\ce{B}],[\ce{F}])-r_{11}([\ce{Y}])-r_{12}([\ce{Y}]).\\
\end{cases}
\end{align}
A Hopf bifurcation at a positive steady state requires the Jacobian of the system \eqref{doubleeq} to possess purely imaginary eigenvalues. Proceeding symbolically, the Jacobian reads:
{\footnotesize
\begin{equation*}
G=
\begin{pmatrix}
-a & 0 & 0 & -e_1 & 0 & u_2 & 0 & 0 & y_{12}\\
0 & -b_4-b_{10} & 0 & -e_4 & -f_{10} & u_3 & v_5 & x_9 & y_{11}\\
0 & 0 & -c & 0 & -f_7 & 0 & v_6 & x_8 & 0\\
-a & -b_4 & 0 & -e_1-e_4 & 0 & u_2+u_3 & v_5+v_6 & 0 & 0\\
0 & -b_{10} & -c & 0 & -f_7-f_{10} & 0 & 0 & x_8+x_9 & y_{11}+y_{12}\\
a & 0 & 0 & e_1 & 0 & -u_2-u_3 & 0 & 0 & 0\\
0 & b_4 & 0 & e_4 & 0 & 0 & -v_5-v_6 & 0 & 0\\
0 & 0 & c & 0 & f_7 & 0 & 0 & -x_8-x_9 & 0\\
0 & b_{10} & 0 & 0 & f_{10} & 0 & 0 & 0 & -y_{11}-y_{12}
\end{pmatrix}.
\end{equation*}}
\hspace{-0.25cm} The partial derivatives are displayed as letters for simplicity. The formal conversion, via $r_{jm}=\partial r_j / \partial [m]$,  reads: 
\begin{equation*}
\begin{alignedat}{4}
a      & = r_{1A}  &\qquad
b_4    & = r_{4B}  &\qquad
b_{10} & = r_{10B} &\qquad
c      & = r_{7C},  \\
e_1    & = r_{1E}  &
e_4    & = r_{4E}  &
f_7    & = r_{7F}  &
f_{10} & = r_{10F}, \\
u_2    & = r_{2U}  &
u_3    & = r_{3U}  &
v_5    & = r_{5V}  &
v_6    & = r_{6V}, \\
x_8    & = r_{8X}  &
x_9    & = r_{9X}  &
y_{11} & = r_{11Y} &
y_{12} & = r_{12Y}.
\end{alignedat}
\end{equation*}
For later reference, we name such parameters $\mathbf{p}$, i.e. 
\begin{equation*}
    \mathbf{p}=(\;a,\;b_4,\;b_{10},\;c,\;e_1,\;e_4,\;f_7,\;f_{10},\;u_2,\;u_3,\;v_5,\;v_6,\;x_8,\;x_9,\;y_{11},\;y_{12}\;).
\end{equation*}
Due to the monotonicity assumption on the reaction rates, we consider all 16 symbols strictly positive, i.e. $\mathbf{p}>0$. The absence of Hopf bifurcation would follow if $G$ does not admit purely imaginary eigenvalues for any choice of $\mathbf{p}$. However, the example in Sec.~\ref{sec:examplehopfpr} shows that such choice exists. Therefore, we restrict the attention to mass action kinetics, which yields the following ODE system,
\begin{equation}\label{doubleeqma}
\begin{cases}
[\dot{\ce{A}}]=-\kappa_1[\ce{A}][\ce{E}]+\kappa_2[\ce{U}]+\kappa_{12}[\ce{Y}],\\
[\dot{\ce{B}}]=\kappa_3[\ce{U}]-\kappa_4[\ce{B}][\ce{E}]
       +\kappa_5[\ce{V}]+\kappa_9[\ce{X}]-\kappa_{10}[\ce{B}][\ce{F}]+\kappa_{11}[\ce{Y}],\\
[\dot{\ce{C}}]=\kappa_6[\ce{V}]-\kappa_7[\ce{C}][\ce{F}]+\kappa_8[\ce{X}],\\
[\dot{\ce{E}}]=-\kappa_1[\ce{A}][\ce{E}]+(\kappa_2+\kappa_3)[\ce{U}]-\kappa_4[\ce{B}][\ce{E}]
       +(\kappa_5+\kappa_6)[\ce{V}],\\
[\dot{\ce{F}}]=-\kappa_7[\ce{C}][\ce{F}] +(\kappa_8+\kappa_9)[\ce{X}]-\kappa_{10}[\ce{B}][\ce{F}]+(\kappa_{11}+\kappa_{12})[\ce{Y}],\\
[\dot{\ce{U}}]=\kappa_1[\ce{A}][\ce{E}]-(\kappa_2+\kappa_3)[\ce{U}],\\
[\dot{\ce{V}}]=\kappa_4[\ce{B}][\ce{E}]-(\kappa_5+\kappa_6)[\ce{V}],\\
[\dot{\ce{X}}]=\kappa_7[\ce{C}][\ce{F}]-(\kappa_8+\kappa_9)[\ce{X}],\\
[\dot{\ce{Y}}]=\kappa_{10}[\ce{B}][\ce{F}]-(\kappa_{11}+\kappa_{12})[\ce{Y}].
\end{cases}
\end{equation}
where $(\kappa_1,...,\kappa_{12})\in\mathbb{R}^{12}_{>0}$ are strictly positive parameters.

This paper proves that the system \eqref{doubleeqma} does not admit Hopf bifurcations at a positive steady state for any choice of the parameters $(\kappa_1,...,\kappa_{12})>0$,  see Thm.~\ref{thm:main}. It is organized as follows. Sec.~\ref{sec:preliminaries} introduces some structural features of the model. Sec.~\ref{sec:examplehopfpr} presents an example of a choice of the parameters $\mathbf{p}$ for the symbolic Jacobian $G$ such that a Hopf bifurcation occurs under parameter-rich kinetics, e.g. Michaelis--Menten \cite{MM13}. We use here the setting of global Hopf bifurcation \cite{Fiedler85PhD, Blokhuis25}, which requires a change of stability of the Jacobian along a steady-state path in parameter space where the Jacobian remains invertible throughout. Based on this example, we state a weaker version of the conjecture, namely that Hopf bifurcation cannot occur from a stable steady state. In Sec.~\ref{sec:sufficient} we derive sufficient algebraic conditions for $G$ to be Hurwitz stable whenever $G$ does not possess a real-positive eigenvalue: such condition depends on the nonnegativity of an expression in the 16 symbols $\mathbf{p}$, which we term $\mathcal{T}$.
In Sec.~\ref{sec:ma} we turn to the mass action case \eqref{doubleeqma} and we prove that - under mass action steady-state constraints - $\mathcal{T}\ge0$. Since we also show that Hopf bifurcations for the mass action system cannot occur if the Jacobian has a real positive eigenvalue, the positivity of $\mathcal{T}$ excludes Hopf bifurcations for the mass action system. The attached \texttt{MATLAB} codes are presented in Sec.~\ref{sec:computations}. In Sec.~\ref{sec:discussionai}, I discuss the proof in more detail from the perspective of my interaction with an \texttt{LLM}. Sec.~\ref{sec:conclusion} concludes the paper with a final wrap-up and few remarks.

\paragraph{Disclaimer on AI contribution.} At each new \texttt{ChatGPT} issue, I have been testing the reasoning advancements by subjecting a few concise and elementary conjectures to the \texttt{LLM}. I own a \textsc{Plus} account, and have exclusively worked with the models therein provided. One of those conjectures is exactly whether the symbolic Jacobian matrix of the dual futile cycle admits purely imaginary eigenvalues for any choice of the positive parameters $\mathbf{p}$. Up to the issue \texttt{GPT-5.6 Sol}, I have always received negative answer - supported by numerical explorations only - but no conclusive proof. Testing this question on \texttt{GPT-5.6 Sol}, in turn, fastly produced the example presented in Sec.~\ref{sec:examplehopfpr}. On the other hand, past independent numerical explorations had already produced a high confidence that, under mass action kinetics, Hopf bifurcation could not occur. This prompted me to understand better the structure of the positive example with general kinetics: it naturally led me to the weaker conjecture \ref{conjecture} (still not proven) that Hopf bifurcation can only possibly occur from an unstable steady state with a real positive eigenvalue - a fact that would already settle the mass action case. Structural investigations of these properties led to Lemma \ref{lem:centralhurwitz} below, which introduces a Hurwitz-type quantity $\mathcal{T}$ whose positivity would settle the conjecture. Such symbolic manipulations have been aided by \texttt{ChatGPT}, but I regard these steps still decisively dependent on a human contribution. After having derived $\mathcal{T}$, I asked \texttt{ChatGPT} to prove its  positivity under mass action kinetics. This produced, in a single query and without further input from my side, a full positivity certificate. To check the computer algebra, I requested \texttt{MATLAB} codes, which I have independently checked. Finally, I wrote the main text without substantial help from an \texttt{LLM}.

\paragraph{Acknowledgments.} I am indebted to Carsten Conradi who first introduced me to this ear-worm problem, and to Khazhgali Kozhasov and Timo de Wolff for early discussions on it and on positivity certificates for polynomials. I further thank Jia-Yuan Dai, Bernold Fiedler, Phillipo Lappicy, Matteo Levi, Alejandro López Nieto, Hannes Stuke, and Stefano Vita for many discussions on doing mathematics with or without AI. This work has been supported by the MATOMIC consortium of the Novo Nordisk Foundation, grant NNF21OC0066551.

\section{Structural preliminaries}\label{sec:preliminaries}

We outline three features of the network \eqref{double}: (i) symmetry, (ii) conserved quantities, and (iii) its interpretation as a quadratic eigenvalue problem.

\paragraph{The $\mathbb{Z}_2$-symmetry $\sigma$.} The network \eqref{double} admits a nontrivial graph automorphism
$\sigma$, which generates a group isomorphic to $\mathbb{Z}_2$. We explicitly record here - for later use - the orbits of such automorphism, at the level of species, reactions, and partial derivatives.
\begin{equation}\label{eq:sigma}\tag{$\sigma$}
    \begin{split}
        \ce{A}\leftrightarrow \ce{C},\qquad
\ce{E}\leftrightarrow \ce{F},\qquad
&\ce{U}\leftrightarrow \ce{X},\qquad
\ce{V}\leftrightarrow \ce{Y},\qquad
\ce{B}\leftrightarrow \ce{B},\\
1\leftrightarrow7,\quad
2\leftrightarrow8,\quad
3\leftrightarrow9&,\quad
4\leftrightarrow10,\quad
5\leftrightarrow11,\quad
6\leftrightarrow12,\\
a\leftrightarrow c,\quad
e_1\leftrightarrow f_7,\quad
u_2\leftrightarrow x_8,\quad
u_3\leftrightarrow x_9,&\quad
b_4\leftrightarrow b_{10},\quad
e_4\leftrightarrow f_{10},\quad
v_5\leftrightarrow y_{11},\quad
v_6\leftrightarrow y_{12}.
\end{split}
\end{equation}
This structural symmetry does not make the system of ODEs $\mathbb{Z}_2$-equivariant \cite{GolubitskySymmBook}, since we do not assume $r_i([m])=r_{\sigma(i)}(\sigma([m]))$. Still, it informally indicates that structural features either come in $\sigma$-related pairs, or they are fixed by $\sigma$.

\paragraph{The three conserved quantities.} Via direct inspection of \eqref{doubleeq}, we note the presence of three conserved quantities, two related by the structural symmetry $\sigma$, i.e.
\begin{equation}\label{eq:conservedquantities1}
    \begin{cases}
    0= [\dot{\ce{E}}]+[\dot{\ce{U}}]+[\dot{\ce{V}}]\qquad\qquad &\Rightarrow \qquad\qquad [\ce{E}]+[\ce{U}]+[\ce{V}]=\mathcal{W}_1\\
0 = [\dot{\ce{F}}]+[\dot{\ce{X}}]+[\dot{\ce{Y}}]\qquad\qquad &\Rightarrow \qquad\qquad [\ce{F}]+[\ce{X}]+[\ce{Y}]=\mathcal{W}_2
\end{cases}
\end{equation}
and one fixed by $\sigma$, i.e.
\begin{equation}\label{eq:conservedquantities2}
        0 =[\dot{\ce{A}}]+ [\dot{\ce{B}}]+ [\dot{\ce{C}}]+ [\dot{\ce{U}}]+ [\dot{\ce{V}}]+[\dot{\ce{X}}]+[\dot{\ce{Y}}] \;\;\Rightarrow  \;\;  [\ce{A}]+ [\ce{B}]+  [\ce{C}]+  [\ce{U}]+ [\ce{V}]+ [\ce{X}]+ [\ce{Y}]=\mathcal{W}_3
\end{equation}
where $\mathcal{W}_1,\mathcal{W}_2,\mathcal{W}_3>0$ are positive constants.  Due to \eqref{eq:conservedquantities1} and \eqref{eq:conservedquantities2}, we know a priori that the characteristic polynomial of $G$ is of the form: 
\begin{equation}\label{eq:charpolyG}
g(\lambda):=\det(\lambda \operatorname{Id} -G)=(\mathfrak{c}_0\lambda^6+\mathfrak{c}_1 \lambda^5+ \mathfrak{c}_2 \lambda^4+\mathfrak{c}_3\lambda^3+\mathfrak{c}_4\lambda^2+\mathfrak{c}_5\lambda+\mathfrak{c}_6)\lambda^3,
\end{equation}
namely $\mathfrak{c}_7=\mathfrak{c}_8=\mathfrak{c}_9\equiv 0$ and it possesses at least three identically zero eigenvalues (we shall see that indeed all other coefficients are non identically zero). Even if $\mathfrak{c}_0=1$, we keep track of this coefficient, for later rescaling reasons. We use the notation $g_{\mathrm{red}}(\lambda)$ for the characteristic polynomial $g(\lambda)$ restricted to a stoichiometric compatibility class:
\begin{equation}
g_{\mathrm{red}}(\lambda):=\mathfrak{c}_0\lambda^6+\mathfrak{c}_1 \lambda^5+ \mathfrak{c}_2 \lambda^4+\mathfrak{c}_3\lambda^3+\mathfrak{c}_4\lambda^2+\mathfrak{c}_5\lambda+\mathfrak{c}_6.
\end{equation}
Consistently, it is possible to interpret $g_{\mathrm{red}}$ as the actual characteristic polynomial of a Jacobian matrix $G_{\mathrm{red}}$ of the system reduced to a stoichiometric compatibility class \cite{Fei19}. To show this, we implement two steps: first we eliminate  the equations for $[\ce{E}]$, $[\ce{F}]$, $[\ce{B}]$ by
\begin{equation}\label{eq:conservedqrem}
\begin{split} [\ce{E}]&=\mathcal{W}_1-[\ce{U}]-[\ce{V}]\\
[\ce{F}]&=\mathcal{W}_2-[\ce{X}]-[\ce{Y}]\\
  [\ce{B}]&=\mathcal{W}_3-[\ce{A}]-[\ce{C}]-[\ce{U}]-[\ce{V}]-[\ce{X}]-[\ce{Y}]
  \end{split}
\end{equation}
Second, we change variables, linearly substituting $\ce{A}$ and $\ce{C}$ with, respectively,
\begin{equation}\label{Qchange}
    [\ce{Q}_1]=[\ce{A}]+[\ce{U}]\qquad\qquad\qquad [\ce{Q}_2]=[\ce{C}]+[\ce{X}],
\end{equation}
Leveraging \eqref{eq:conservedqrem} and \eqref{Qchange}, the system restricted to $([\ce{Q_1}],[\ce{Q_2}],[\ce{U}],[\ce{V}],[\ce{Y}],[\ce{X}])$ reads:
\begin{equation*}
\begin{cases}
    [\dot{\ce{Q}}_1]=-r_3([\ce{U}])+r_{12}([\ce{Y}])\;\\
    [\dot{\ce{Q}}_2]=r_6([\ce{V}])-r_9([\ce{X}])\;\\
    [\dot{\ce{U}}]=r_{1}([\ce{Q_1}]-[\ce{U}],\mathcal{W}_1-[\ce{U}]-[\ce{V}])-r_2([\ce{U}])-r_3([\ce{U}])\;\\
[\dot{\ce{V}}]=r_4(\mathcal{W}_3-[\ce{Q_1}]-[\ce{Q_2}]-[\ce{V}]-[\ce{Y}],\mathcal{W}_1-[\ce{U}]-[\ce{V}])-r_5([\ce{V}])-r_6([\ce{V}])\;\\
[\dot{\ce{Y}}]=r_{10}(\mathcal{W}_3-[\ce{Q_1}]-[\ce{Q_2}]-[\ce{V}]-[\ce{Y}],\mathcal{W}_2-[\ce{X}]-[\ce{Y}])-r_{11}([\ce{Y}])-r_{12}([\ce{Y}])\;\\
[\dot{\ce{X}}]=r_7([\ce{Q_2}]-[\ce{X}],\mathcal{W}_2-[\ce{X}]-[\ce{Y}])-r_8([\ce{X}])-r_9([\ce{X}])\;\\
\end{cases}.
\end{equation*}
The placement of $\ce{Y}$ before $\ce{X}$ here is intentional to expose a clearer structure, as we will see below. Its Jacobian $G_{\mathrm{red}}$, via chain rule, is written as
{\footnotesize \begin{equation*}
G_{\mathrm{red}}=
    \begin{pmatrix}
        0 & 0 & -u_3 & 0 & y_{12} & 0\\
        0 & 0 & 0 & v_6 & 0 &-x_9\\
        a & 0 & -a-e_1-u_2-u_3 & -e_1 & 0 & 0\\
        -b_4 & -b_4 & -e_4 & -b_4-e_4-v_5-v_6 & -b_4 & 0\\
        -b_{10} & -b_{10} & 0 & -b_{10} & -b_{10}-f_{10}-y_{11}-y_{12} & -f_{10}\\
        0 & c & 0 & 0& -f_7 & -c-f_7-x_8-x_9
    \end{pmatrix}.
\end{equation*}}
\hspace{-0.25cm} The polynomial $g_{\mathrm{red}}$ is indeed the characteristic polynomial of the reduced Jacobian $G_{\mathrm{red}}$. The change of variables from $([\ce{A}],[\ce{C}])$ to $([\ce{Q}_1],[\ce{Q}_2])$ is particularly convenient to expose a relevant structure:  we split $G_{\mathrm{red}}$ into blocks as follows:
\begin{equation}
    G_{\mathrm{red}}=\begin{pmatrix}
        \mathbf{0}_2 & R\\
        L & -N
    \end{pmatrix},
\end{equation}
with 
{\small \[
    \mathbf{0}_2=\begin{pmatrix} 0 & 0\\
        0 &0\end{pmatrix},\qquad\qquad\qquad R=\begin{pmatrix}
           -u_3 & 0 & y_{12} & 0\\
           0 & v_6 & 0 &-x_9\\
        \end{pmatrix}, \qquad\qquad\qquad
        L=\begin{pmatrix}
             a & 0\\
             -b_4 & -b_4\\
              -b_{10} & -b_{10} \\
               0 & c
        \end{pmatrix},\]}
        \[
        N=\begin{pmatrix}
a+e_1+u_2+u_3 & e_1 & 0 & 0\\
e_4 & b_4+e_4+v_5+v_6 & b_4 & 0\\
0 & b_{10} & b_{10}+f_{10}+y_{11}+y_{12} & f_{10}\\
0 & 0& f_7 & c+f_7+x_8+x_9
        \end{pmatrix}.
\]
The matrix $N$ is row-diagonally dominant with $N_{ii}>0$; thus all of the eigenvalues of $-N$ have negative real part. The choice of placing the equation for $\ce{Y}$ before the one for $\ce{X}$ makes this matrix tridiagonal.

\paragraph{The quadratic eigenvalue problem.} The structure described above plays a central role in this paper, especially for the mass-action case. We transform $N$ in a symmetric matrix via left multiplication by a $4\times 4$ positive diagonal matrix $D=\operatorname{diag}(\delta_1,\delta_2,\delta_3,\delta_4)$
defined as follows:
\begin{equation}D=\operatorname{diag}\bigg(\delta_1,\delta_2,\delta_3,\delta_4\bigg):=\operatorname{diag}\bigg(\dfrac{e_4}{e_1},1,\dfrac{b_4}{b_{10}},\dfrac{b_4f_{10}}{b_{10}f_{7}}\bigg).
\end{equation}
The ratios $(e_4/e_1,b_4/b_{10},f_{10}/f_7)$ express the relative magnitude of the contribution of $\ce{E}$, $\ce{B}$, $\ce{F}$ in reactions, respectively, $(4, 1)$, $(4,10)$, and $(10,7)$. Arbitrarily, the matrix is normalized with respect to $D_{22}=\delta_2=1$
We may further expose the role of the diagonally-dominant symmetric matrix 
{\footnotesize\[DN= 
         \begin{pmatrix}
            \dfrac{e_4}{e_1} (a+e_1+u_2+u_3) & e_4 & 0 & 0\\
            e_4 & b_4+e_4+v_5+v_6& b_4 &0\\
             0 & b_4 & \dfrac{b_4}{b_{10}}(b_{10}+f_{10}+y_{11}+y_{12}) & \dfrac{b_4f_{10}}{b_{10}}\\
             0 & 0 & \dfrac{b_4f_{10}}{b_{10}} & \dfrac{b_4f_{10}}{b_{10}f_{7}}(c+f_7+x_8+x_9)
         \end{pmatrix}.\]}
         \hspace{-0.25cm} Multiplication of the reduced characteristic polynomial $g_{\mathrm{red}}(\lambda)$ by $\det D=e_4b_4^2f_{10}/e_1b_{10}^2f_{7}$ and working via Schur's complement, yields:
\[
\begin{split}
\frac{e_4b_4^2f_{10}}{e_1b_{10}^2f_{7}}\,g_{\mathrm{red}}(\lambda)
&=\det(D)\,\det(\lambda\operatorname{Id}_6-G_{\mathrm{red}})\\
&=
\det(D)\,
\det\begin{pmatrix}
\lambda\operatorname{Id}_2&-R\\
-L&\lambda\operatorname{Id}_4+N
\end{pmatrix}\\
&=
\det(D)\,\lambda^2
\det\left[
(\lambda\operatorname{Id}_4+N)
-\frac{1}{\lambda}LR
\right]\\
&=
\lambda^2
\det\left[
\lambda D+ DN
-\frac{1}{\lambda} DLR
\right]\\
&=
\lambda^{-2}
\det\left[
\lambda^2 D
+\lambda DN
-DLR
\right].
\end{split}
\]
Equivalently,
\begin{equation}\label{eq:qep}
    \lambda^2\frac{e_4b_4^2f_{10}}{e_1b_{10}^2f_{7}} \,g_{\mathrm{red}}=\det (
\lambda^2 D
+\lambda DN
- DLR),
\end{equation}
where
     \begin{equation}
         DL=\begin{pmatrix}
            \dfrac{e_4}{e_1}a & 0\\
            -b_4 & -b_4\\
            -b_4 & -b_4\\
            0 & \dfrac{b_4f_{10}}{b_{10}f_{7}}c
         \end{pmatrix}.
     \end{equation}
This transformation views the characteristic polynomial $g_{\mathrm{red}}(\lambda)$ of the system as a quadratic eigenvalue problem \cite{Tisseur:01}. This would be true even without the diagonal rescale $D$. However, since both $D$ and $DN$ are Hermitians and positive definite, such construction identifies the rank-2 coupling $-DLR$ as the only possible culprit of any instability. Indeed: if $-DLR$ were Hermitian and positive semidefinite, then all eigenvalues $\lambda$ would satisfy $\Re (\lambda) \le0$.

%\hline

\section{An example of a Hopf bifurcation and a weaker conjecture}\label{sec:examplehopfpr}

The structure of $G$ itself allows for purely-imaginary crossing, although for a very inhomogeneous choice of parameters. Specifically, we can confirm this as follows. 

Fix $a=b_4=c=f_7=f_{10}=e_1=u_2=u_3=v_5=x_8=x_9=y_{11}=y_{12}=1$, $b_{10}=0.01$, $e_4=4\cdot10^4$, and let $v_6$ be a free parameter. The evaluated matrix $\overline{G}_{\mathrm{red}}$ gives
\begin{equation}\label{Jacobianexample}
\overline{G}_{\mathrm{red}}(v_6)=
\begin{pmatrix}
0 & 0 & -1 & 0 & 1 & 0\\
0 & 0 & 0 & v_6 & 0 & -1\\
1 & 0 & -4 & -1 & 0 & 0\\
-1 & -1 & -40000 & -(40002+v_6) & -1 & 0\\
-0.01 & -0.01 & 0 & -0.01 & -3.01 & -1\\
0 & 1 & 0 & 0 & -1 & -4
\end{pmatrix},
\end{equation}
whose determinant reads:
$$\det \overline{G}_{\mathrm{red}}(v_6)=\dfrac{6080453}{50} - \dfrac{2079547\;v_6}{50}\approx 121609-41591 \;v_6,$$
from which we can conclude that $\det \overline{G}_{\mathrm{red}}(v_6)$ is monotonically decreasing w.r.t. $v_6$, and a zero occurs somewhere around $v_6\approx2.92$. Due to monotonicity, in particular, on the interval e.g. $v_6\in[5,40000]$ we have 
$$\det \overline{G}_{\mathrm{red}}(v_6)<0,$$
and thus $\overline{G}_{\mathrm{red}}$ is invertible on this interval. However, explicit eigenvalues computation shows that at the endpoints we have 
\begin{equation}
    \begin{cases}
       \operatorname{eig}( \overline{G}_{\mathrm{red}}(5))\approx (\mathbf{+8.52 \cdot 10^{-2}}, -1.28 \pm \mathrm{i}\; 0.44, -3.15,-4.38, -4 \cdot 10^{4})\,,\\
        \operatorname{eig}(\overline{G}_{\mathrm{red}}(40000))\approx(\mathbf{+5.32}, \mathbf{+0.01 \pm \mathrm{i}6.84}, -8.18\pm \mathrm{i}\,4.08,-8\cdot 10^4) \,.
    \end{cases}
\end{equation}
In particular, by intermediate value theorem, the invertibility of $\overline{G}_{\mathrm{red}}$ yields that there exists at least one $v_6^*\in (5,40000)$ such that $\overline{G}_{\mathrm{red}}(v_6^*)$ has purely-imaginary eigenvalues. If the system is endowed with \emph{parameter-rich kinetics} \cite{VasStad23}, as e.g. Michaelis--Menten \cite{MM13} or Generalized Mass Action kinetics \cite{Muller:12}, any choice of the sixteen partial-derivatives symbols is realizable at a steady state, by independent parametrization.
We can then leverage the theory of global Hopf bifurcation \cite{Fiedler85PhD}, via Recipe 0 in \cite{Blokhuis25}, which implies the existence of nonstationary periodic orbits for some values of $v_6$ for any analytic kinetics (including in particular Michaelis--Menten and Generalized Mass Action). Note the Recipe 0 is stated in \cite{Blokhuis25} for the main case of interest, namely a transition from a stable to an unstable steady state, whereas here the spectrum changes within the unstable reason. The underlying argument based on \cite{Fiedler85PhD} is unchanged. Without further local analysis, this global approach does not determine whether the Hopf bifurcation is nondegenerate, nor in particular whether it is subcritical or supercritical. Since the bifurcation occurs from an unstable steady state, the resulting periodic orbit will be unstable in any case, at least locally. We do not discuss this approach further, as it falls beyond the direct scope of this paper, and refer to \cite{Blokhuis25} for more details.

Independently from nonlinear dynamics, we have shown that purely imaginary eigenvalues of $G_{\mathrm{red}}$ (or equivalently of $G$) may exist for some choice of the unconstrained sixteen positive parameters. 
However, we observe that the bifurcation parameter $v_6$ triggers in an ordered (!) sequence a zero eigenvalue bifurcation and afterwards a purely-imaginary bifurcation. I conjecture that no Hopf bifurcation can occur before the zero-eigenvalue one. A bit more strongly, I conjecture the following:

\begin{conjecture}\label{conjecture}
    Assume $\operatorname{sign}\det G_\mathrm{red}>0$, then $G_\mathrm{red}$ is Hurwitz-stable.
\end{conjecture}

The conjecture would directly settle also the mass action case: $\operatorname{sign}\det G_\mathrm{red}>0$ is a necessary condition for the occurrence of purely imaginary eigenvalues under mass action kinetics as we will see in Sec.~\ref{sec:exclusion}.
Unfortunately, I have been unable to prove the above conjecture. In the next section, we still pursue this route as it eventually leads to a more favorable sufficient condition to exclude Hopf bifurcation under mass action kinetics.

\section{A sufficient condition for Hurwitz stability}\label{sec:sufficient}

Let $(\mathfrak{c}_0,\mathfrak{c}_1,...,\mathfrak{c}_6)$ indicate the nontrivial coefficients of the characteristic polynomial of the reduced symbolic Jacobian $G$
\begin{equation}
g_{\mathrm{red}}(\lambda):=\det(\lambda \operatorname{Id} -G_{\mathrm{red}})=\mathfrak{c}_0\lambda^6+\mathfrak{c}_1 \lambda^5+ \mathfrak{c}_2 \lambda^4+\mathfrak{c}_3\lambda^3+\mathfrak{c}_4\lambda^2+\mathfrak{c}_5\lambda+\mathfrak{c}_6
\end{equation}
Conjecture \ref{conjecture} assumes  that $\mathfrak{c}_6=\det G_{\mathrm{red}}>0$. We investigate this assumption further.

 Via the Cauchy--Binet formula, every coefficient $\mathfrak{c}_i$ is a multilinear homogeneous polynomial in the sixteen parameters of degree $i$.
Besides trivial $\mathfrak{c}_0=1$, an explicit computation of the coefficients show that $\mathfrak{c}_1,\mathfrak{c}_2,\mathfrak{c}_3,\mathfrak{c}_4>0$, i.e. all summands have positive coefficient, while $\mathfrak{c}_5$ and $\mathfrak{c}_6$ have summands with both positive and negative coefficients. To underline this, we introduce the notation:
\begin{equation}
    \begin{cases}
        \mathfrak{c}_5=P_5-N_5 \;\\
        \mathfrak{c}_6=P_6-N_6,\\
    \end{cases}
\end{equation}
where $P_5,N_5,P_6,N_6>0$ are multilinear homogenous (of degree 5 and 6) polynomials in the sixteen parameters of $G$. It is worthwhile make the negative terms explicit:
\begin{equation}\label{eq:N56}
    \begin{cases}
        N_5=v_6y_{12}(b_4cf_{10}+ab_{10}e_4)\\
        N_6=v_6y_{12}\Big((a+e_1+u_2+u_3) b_4cf_{10}+(c+f_7+x_8+x_9)ab_{10}e_4+ace_4f_{10}
\Big).
\\
    \end{cases}
\end{equation}
In particular either of the three different combinations $b_4cf_{10}v_6y_{12}$, $ab_{10}e_4v_6y_{12}$, (related by the symmetry $\sigma$) and $ace_4f_{10}v_6y_{12}$ (fixed by $\sigma)$ appears in all negative monomials of any coefficient of the characteristic polynomial. More specifically, the quadratic term $v_6y_{12}$ is present in all of the negative summands. We underline that $v_6$ and $y_{12}$ are indeed related by the symmetry $\sigma$.  These three terms structurally correspond to the \emph{unstable cores} introduced in \cite{VasStad23}.

A matrix is \emph{Hurwitz-stable} if all of its eigenvalues have negative real part, and it is \emph{Hurwitz-unstable} if at least one of its eigenvalues has positive real part. We may study the Hurwitz stability of a matrix via its associated Hurwitz matrix, which reads in our case:
\[
\mathcal{H}_6=
\begin{pmatrix}
\mathfrak{c}_1 & \mathfrak{c}_3 & \mathfrak{c}_5 & 0   & 0   & 0\\
\mathfrak{c}_0   & \mathfrak{c}_2 & \mathfrak{c}_4 & \mathfrak{c}_6 & 0   & 0\\
0   & \mathfrak{c}_1 & \mathfrak{c}_3 & \mathfrak{c}_5 & 0   & 0\\
0   & \mathfrak{c}_0   & \mathfrak{c}_2 & \mathfrak{c}_4 & \mathfrak{c}_6 & 0\\
0   & 0   & \mathfrak{c}_1 & \mathfrak{c}_3 & \mathfrak{c}_5 & 0\\
0   & 0   & \mathfrak{c}_0   & \mathfrak{c}_2 & \mathfrak{c}_4 & \mathfrak{c}_6
\end{pmatrix}.
\]
The Hurwitz determinants $\Delta_i$ are the leading principal minors of $\mathcal{H}_6$. Direct expansion yields:
\begin{equation}
\begin{cases}
\Delta_1=\mathfrak{c}_1\,;\\
\Delta_2=\mathfrak{c}_1\mathfrak{c}_2-\mathfrak{c}_0\mathfrak{c}_3\,;\\
\Delta_3=\mathfrak{c}_3\Delta_2-\mathfrak{c}_1^2\mathfrak{c}_4+\mathfrak{c}_0\mathfrak{c}_1\mathfrak{c}_5\,;\\
\Delta_4
=\mathfrak{c}_4(\Delta_3-\mathfrak{c}_0\mathfrak{c}_1\mathfrak{c}_5)
-\mathfrak{c}_5(\mathfrak{c}_2\Delta_2-2\mathfrak{c}_0\mathfrak{c}_1\mathfrak{c}_4)
-\mathfrak{c}_0^2\mathfrak{c}_5^2+\mathfrak{c}_1\mathfrak{c}_6\Delta_2\,;\\
\Delta_5
=\mathfrak{c}_5\Delta_4
-\mathfrak{c}_6(\mathfrak{c}_3\Delta_3
-\mathfrak{c}_1\mathfrak{c}_5\Delta_2
+\mathfrak{c}_1^3\mathfrak{c}_6)\,;\\
\Delta_6=\mathfrak{c}_6 \Delta_5\;.
\end{cases}
\end{equation}
Hurwitz-stability is characterized by $\Delta_i>0$ for $i=1,...,6$.
It is also well-known that nonzero purely-imaginary eigenvalues force $ \Delta_{n-1}= \Delta_5=0 $. This determinant corresponds, up to sign convention, to the determinant of the so-called \emph{second additive compound}, which is equal to the product of pairwise sums of the eigenvalues, see for an overview in reaction networks the preprint by Banaji \cite{banaji:2ndaddcomp}. This can also be recovered in general by standard resultant arguments \cite{householder:1968}.

Preliminarily, a universal algebraic relation on the Hurwitz determinants is needed, as stated in the next Proposition.

\begin{prop}\label{prop:Hdelta25}
Let $\Delta_2,\Delta_3,\Delta_4,\Delta_5$ be as introduced above. It holds
\begin{equation}\label{HurIddelta25}
\Delta_2\Delta_5
=
(\mathfrak{c}_5\Delta_2-\mathfrak{c}_1^2\mathfrak{c}_6)\Delta_4
-\mathfrak{c}_6\Delta_3^2,
\end{equation}
\end{prop}
\proof 
Introduce  $B_1^+
:=
\mathfrak{c}_3\Delta_3-\mathfrak{c}_1\mathfrak{c}_5\Delta_2+\mathfrak{c}_1^3\mathfrak{c}_6$ so that $\Delta_5=\mathfrak{c}_5\Delta_4-\mathfrak{c}_6B_1^+$, and derive:
\begin{equation}\label{eq:47}
\begin{split}
\Delta_2 B_1^+
&=\mathfrak{c}_3\Delta_2\Delta_3
-\mathfrak{c}_1\mathfrak{c}_5\Delta_2^2
+\mathfrak{c}_1^3\mathfrak{c}_6\Delta_2\\
&=(\Delta_3+\mathfrak{c}_1^2\mathfrak{c}_4-\mathfrak{c}_0\mathfrak{c}_1\mathfrak{c}_5)\Delta_3
-\mathfrak{c}_1\mathfrak{c}_5\Delta_2^2
+\mathfrak{c}_1^3\mathfrak{c}_6\Delta_2\\
&=\Delta_3^2
+\mathfrak{c}_1[
\mathfrak{c}_1\mathfrak{c}_4\Delta_3
-\mathfrak{c}_5(\mathfrak{c}_0\Delta_3+\Delta_2^2)
+\mathfrak{c}_1^2\mathfrak{c}_6\Delta_2
]\\
&=\Delta_3^2
+\mathfrak{c}_1[
\mathfrak{c}_1\mathfrak{c}_4\Delta_3
-\mathfrak{c}_5(\mathfrak{c}_0\Delta_3+\mathfrak{c}_1\mathfrak{c}_2\Delta_2-\mathfrak{c}_0\mathfrak{c}_3\Delta_2)
+\mathfrak{c}_1^2\mathfrak{c}_6\Delta_2
]\\
&=\Delta_3^2
+\mathfrak{c}_1[
\mathfrak{c}_1\mathfrak{c}_4\Delta_3
-\mathfrak{c}_5(\mathfrak{c}_0\Delta_3+\mathfrak{c}_1\mathfrak{c}_2\Delta_2-\mathfrak{c}_0(\Delta_3+\mathfrak{c}_1^2\mathfrak{c}_4-\mathfrak{c}_0\mathfrak{c}_1\mathfrak{c}_5))
+\mathfrak{c}_1^2\mathfrak{c}_6\Delta_2
]\\
&=\Delta_3^2
+\mathfrak{c}_1^2[
\mathfrak{c}_4\Delta_3
-\mathfrak{c}_5(\mathfrak{c}_2\Delta_2-\mathfrak{c}_0\mathfrak{c}_1\mathfrak{c}_4+\mathfrak{c}_0^2\mathfrak{c}_5)
+\mathfrak{c}_1\mathfrak{c}_6\Delta_2
]\\
&=\Delta_3^2+\mathfrak{c}_1^2\Delta_4.
\end{split}
\end{equation}
To show \eqref{HurIddelta25}, we compute: 
\[
\begin{split}
\Delta_2\Delta_5&=\mathfrak{c}_5\Delta_2\Delta_4
-\mathfrak{c}_6\Delta_2B^+_1\\
&=\mathfrak{c}_5\Delta_2\Delta_4
-\mathfrak{c}_6(\Delta_3^2+\mathfrak{c}_1^2\Delta_4)\\
&=(\mathfrak{c}_5\Delta_2-\mathfrak{c}_1^2\mathfrak{c}_6)\Delta_4-\mathfrak{c}_6\Delta_3^2.
\end{split}\]
\endproof

The following technical Proposition is needed. 

\begin{prop}[Positive building blocks.]\label{prop:pbb}
The following inequalities hold:
\begin{enumerate}
\item $N_6P_5-N_5P_6>0$
    \item $\mathfrak{c}_i>0$ for $i=0,...,4$;
    \item $\Delta_i >0$ for $i=1,...,3$;
    \item if $\mathfrak{c}_6>0$, then $\mathfrak{c}_5>0$
    \item if $\Delta_4>0$ then
    \begin{equation}
        B_1^+
:=
\mathfrak{c}_3\Delta_3-\mathfrak{c}_1\mathfrak{c}_5\Delta_2+\mathfrak{c}_1^3\mathfrak{c}_6>0
    \end{equation}
    \item 
    \[B_2^+
:= P_5\left(
\mathfrak{c}_2\Delta_2
-2\mathfrak{c}_0\mathfrak{c}_1\mathfrak{c}_4
+\mathfrak{c}_0^2(P_5+\mathfrak{c}_5)
\right)
-\mathfrak{c}_1P_6\left(
\mathfrak{c}_1\mathfrak{c}_2-2\mathfrak{c}_0\mathfrak{c}_3
\right)>0\].
\end{enumerate}
\end{prop}
\proof We prove the statements in sequence. Substitute $\mathfrak{c}_0=1$ throughout.

\begin{enumerate}
    \item Explicit expansion of $N_6P_5-N_5P_6$ yields a degree 11  homogeneous polynomial with 1588 monomials, all with positive coefficient.
    \item The coefficients $\mathfrak{c}_0,\mathfrak{c}_1,..., \mathfrak{c}_4>0$ are positive: it follows from explicit expansion. 
    \item $\Delta_1=\mathfrak{c}_1>0$ is trivial. The positivity of $\Delta_2$ and $\Delta_3$ is achieved by explicit expansion: $\Delta_2$ consists of 746 monomials with positive coefficients, while $\Delta_3$ consists of 41,243 monomials with positive coefficients.
    \item For $i=5,6$, $\mathfrak{c}_i>0$ is equivalent to $P_i>N_i$. Then, $\mathfrak{c}_6>0$ yields
    \[N_5P_6>N_5N_6,\]
    and combining with $N_6P_5>N_5P_6$ already proved, yields
    \[N_6P_5>N_5N_6,\]
    which proves $\mathfrak{c}_5>0$ via division by $N_6$.

\item It follows via the expansion \eqref{eq:47} above, from positivity of $\Delta_2$ proven above and of $\Delta_4$ assumed.
\item From the expansion of $\Delta_3$, it follows the following algebraic equality:
\[ 
\begin{split}
    \Delta_2(\mathfrak{c}_1\mathfrak{c}_2-2\mathfrak{c}_3)+2(\Delta_3-\mathfrak{c}_1\mathfrak{c}_5)&=\Delta_2(\mathfrak{c}_1\mathfrak{c}_2-2\mathfrak{c}_3)+2\mathfrak{c}_3\Delta_2-2\mathfrak{c}_1^2\mathfrak{c}_4\\&=\mathfrak{c}_1(\mathfrak{c}_2\Delta_2-2\mathfrak{c}_1\mathfrak{c}_4),
\end{split}
\]
which yields:
\[
\begin{split}
\mathfrak{c}_1B^+_2=&\mathfrak{c}_1P_5(\mathfrak{c}_2\Delta_2-2\mathfrak{c}_1\mathfrak{c}_4+P_5+\mathfrak{c}_5)
-\mathfrak{c}_1^2P_6(\mathfrak{c}_1\mathfrak{c}_2-2\mathfrak{c}_3)\\
=&P_5[\mathfrak{c}_1(\mathfrak{c}_2\Delta_2-2\mathfrak{c}_1\mathfrak{c}_4)+\mathfrak{c}_1(P_5+\mathfrak{c}_5)]-\mathfrak{c}_1^2P_6(\mathfrak{c}_1\mathfrak{c}_2-2\mathfrak{c}_3)\\
=&P_5[\Delta_2(\mathfrak{c}_1\mathfrak{c}_2-2\mathfrak{c}_3)+2(\Delta_3-\mathfrak{c}_1\mathfrak{c}_5)+\mathfrak{c}_1(P_5+\mathfrak{c}_5)]-\mathfrak{c}_1^2P_6(\mathfrak{c}_1\mathfrak{c}_2-2\mathfrak{c}_3)\\
=&(P_5\Delta_2-\mathfrak{c}_1^2P_6)(\mathfrak{c}_1\mathfrak{c}_2-2\mathfrak{c}_3)+P_5[2\Delta_3-\mathfrak{c}_1(\mathfrak{c}_5-P_5)]\\
=&(P_5\Delta_2-\mathfrak{c}_1^2P_6)(\mathfrak{c}_1\mathfrak{c}_2-2\mathfrak{c}_3)+2P_5\Delta_3+\mathfrak{c}_1P_5N_5.
\end{split}\]
\end{enumerate}
Explicit computation of $(P_5\Delta_2-\mathfrak{c}_1^2P_6)$ and $(\mathfrak{c}_1\mathfrak{c}_2-2\mathfrak{c}_3)$ shows that these expressions are both positive, respectively with 72,468 positive  and 744 positive monomials. Positivity of $\mathfrak{c}_1, \Delta_3, P_5, N_5$ concludes for $B^+_2>0$. The necessary explicit expansions are implemented in the attached \texttt{MATLAB} script \texttt{dualfutile\_positivity\_checks.m}.
\endproof

The next lemma states path-connectivity of the parameter region where $\mathfrak{c}_6>0$.

\begin{lemma}\label{lem:c6pos}
Let $\mathcal{C}^+_6\subset \mathbb{R}^{16}_{>0}$ be the parameter region where $\mathfrak{c}_6>0$, i.e.
\begin{equation}
    \mathcal{C}^+_6:=\{(a,b_4,b_{10},c,e_1,e_4,f_7,f_{10},u_2,u_3,v_5,v_6,x_8,x_9,y_{11},y_{12})\in\mathbb{R}^{16}_{>0}\;|\;\mathfrak{c}_6>0\}.
\end{equation}
Then $\mathcal{C}^+_6$ is path-connected.
\end{lemma}
\begin{proof}
    For any two parameter points in $\mathcal{C}^+_6$, we construct explicitly a path within $\mathcal{C}^+_6$ connecting them. The main structural observation is that all the negative terms in $\mathfrak{c}_6$ contain the parameters $v_6y_{12}$, see \eqref{eq:N56}. We utilize $v_6$, but a $\sigma$-symmetric argument holds via $y_{12}$. For notation, let $\mathbf{p}^{\vee v_6}=(a,b_4,b_{10},c,e_1,e_4,f_7,f_{10},u_2,u_3,v_5,x_8,x_9,y_{11},y_{12})$ indicate the fifteen parameters, excluded $v_6$, and let $(\mathbf{p}^{\vee v_6}_1,v_6^{(1)})$, $(\mathbf{p}^{\vee v_6}_2,v_6^{(2)})$ be any two points in $\mathcal{C}^+_6$. Let $t\in[0,3]$, we construct a path $p(t)$ from $(\mathbf{p}^{\vee v_6}_1,v_6^{(1)})$ to $(\mathbf{p}^{\vee v_6}_2,v_6^{(2)})$ in three steps as follows: 
    \begin{equation}
        p(t,\varepsilon)=\begin{cases}
            (\mathbf{p}^{\vee v_6}_1,(1-t)v_6^{(1)}+t \varepsilon)\quad&\text{for $t\in[0,1]$};\\
            ((2-t)\mathbf{p}^{\vee v_6}_1+(t-1)\,\mathbf{p}^{\vee v_6}_2,\varepsilon)\quad&\text{for $t\in[1,2]$};\\
            (\mathbf{p}^{\vee v_6}_2,(3-t)\varepsilon+(t-2)\,v_6^{(2)})\quad&\text{for $t\in[2,3]$,}
        \end{cases}
    \end{equation}
where the notation $t\,\mathbf{p}^{\vee v_6}$ is intended component-wise. Clearly, any such $p(t,\varepsilon)$, for $\varepsilon>0$ is a path in $\mathbb{R}^{16}_{>0}$. Moreover, $\mathfrak{c}_6>0$ along $p(t,\varepsilon)$  with $\varepsilon=0$ in $t\in[1,2]$, since $N_6(\varepsilon=0)\equiv0$ and $P_6(\varepsilon=0)>0$. Due to linearity of $v_6$ in $\mathfrak{c}_6$, it follows that $p(t,\varepsilon)\in \mathcal{C}^+_{6}$ for $t\in[0,1]\cup[2,3]$ and any $0<\varepsilon <\operatorname{min}(v_6^{(1)},v_6^{(2)})$. Moreover, by continuity and compactness of $[1,2]$, there exists $\bar{\varepsilon}>0$ small enough such that $p(t,\bar{\varepsilon})\in \mathcal{C}^+_6$ and the lemma is proven.
\end{proof}

Now, we get to the main Lemma of this section. 

\begin{lemma}\label{lem:centralhurwitz}
    Assume $\mathfrak{c}_6>0$ and $\Delta_4>0$. If 
    \begin{equation}\label{eq:star}
       \mathcal{T}:=P_5\Delta_4-P_6B^+_1-N_5B^+_2\ge 0,
    \end{equation}
    then $\Delta_5>0$.
\end{lemma}

\proof
We recall the split $\mathfrak{c}_5=P_5-N_5$ and $\mathfrak{c}_6=P_6-N_6$ with $P_5,N_5,P_6,N_6>0$, and positivity of $B^+_2$, $(N_6P_5-N_5P_6)$ (Prop.~\ref{prop:pbb}), $\Delta_4$ (by assumption), $\mathfrak{c}_6$ (by assumption), and the nonnegativity of $P_5\Delta_4-P_6B^+_1-N_5B^+_2\ge0$ (by assumption). We get:
\begin{equation}\label{eq:TfromDelta5}
    \begin{split}
P_6\Delta_5&=P_6(\mathfrak{c}_5\Delta_4
-\mathfrak{c}_3\mathfrak{c}_6\Delta_3
+\mathfrak{c}_1\mathfrak{c}_5\mathfrak{c}_6\Delta_2
-\mathfrak{c}_1^3\mathfrak{c}_6^2)\\
&=P_6(\mathfrak{c}_5\Delta_4 - \mathfrak{c}_6B_1^+)\\
&=\mathfrak{c}_5P_6\Delta_4-\mathfrak{c}_6P_6B^+_1-\mathfrak{c}_6P_5\Delta_4+\mathfrak{c}_6P_5\Delta_4\\
&=(\mathfrak{c}_5P_6-\mathfrak{c}_6P_5)\Delta_4+\mathfrak{c}_6(P_5\Delta_4-P_6B_1^+)\\
&=(P_5P_6-N_5P_6-P_6P_5+N_6P_5)\Delta_4+\mathfrak{c}_6(P_5\Delta_4-P_6B_1^+)\\
&=(N_6P_5-N_5P_6)\Delta_4+\mathfrak{c}_6(P_5\Delta_4-P_6B_1^+)\\
&>\mathfrak{c}_6(P_5\Delta_4-P_6B_1^+)\\
&>\mathfrak{c}_6(P_5\Delta_4-P_6B_1^+-N_5B^+_2)\\
&\ge 0.
\end{split}
\end{equation}
And the statement follows.
\endproof

Lemma~\ref{lem:centralhurwitz} yields the following corollary, which marks my closest advancement to Conjecture \ref{conjecture}.

\begin{cor}
    Assume $\mathcal{T}\ge0$ on $\mathcal{C}^+_6$. Then $G_{\mathrm{red}}$ is Hurwitz stable throughout $\mathcal{C}^+_6$.
\end{cor}
\proof
The statement follows combining the two previous Lemma \ref{lem:c6pos} and Lemma \ref{lem:centralhurwitz}. Indeed, indirectly assume that there exists a parameter point $\mathbf{p}_{unst}$ in $\mathcal{C}^+_6$ whose evaluated $G_{\mathrm{red}}$ is not Hurwitz-stable. Let $\mathbf{p}_{st}$ be any parameter point where $G_{\mathrm{red}}$ is Hurwitz stable: the existence of such points can be easily found explicitly (and it is well-known in the literature). Consider a path $p(t)$ connecting $\mathbf{p}_{st}$ to $\mathbf{p}_{unst}$, and let $t^*$ the first point where $G_{\mathrm{red}}(p(t^*))$ loses stability. Since $\mathfrak{c}_6>0$ throughout, such point must identify purely imaginary eigenvalues, and in particular satisfying $\Delta_5=0$. However, since $t^*$ is the first point with a loss of stability, it follows that $\Delta_4>0$ for $t<t^*$ and thus - by continuity - $\Delta_4\ge0$ at $t^*$. The equality \eqref{HurIddelta25} excludes $\Delta_4=0$ under $\Delta_3>0$ (by Prop.~\ref{prop:pbb}), $\Delta_5=0$ (by indirect assumption), and $\mathfrak{c}_6>0$ (by assumption), thus $\Delta_4>0$ at $t^*$, and Lemma \ref{lem:centralhurwitz} leads to contradiction.
\endproof

While I was unable to prove $\mathcal{T}\ge0$ for the 16 general independent parameters $\mathbf{p}$, it is possible to prove the nonnegativity of $\mathcal{T}$ under mass action kinetics. Such arguments are collected in the next section. 

\section{Mass action kinetics}\label{sec:ma}

\subsection{Steady-state constraints}\label{sec:massc}
At a positive steady state, the mass action system reads:
\begin{align}
\kappa_1[\ce{A}][\ce{E}]&=\kappa_2[\ce{U}]+\kappa_{12}[\ce{Y}],\label{A-eq}\tag{A-eq}\\
\kappa_4[\ce{B}][\ce{E}]+\kappa_{10}[\ce{B}][\ce{F}]&=\kappa_3[\ce{U}]+\kappa_{11}[\ce{Y}]
       +\kappa_5[\ce{V}]+\kappa_9[\ce{X}],\label{B-eq}\tag{B-eq}\\
       \kappa_7[\ce{C}][\ce{F}]&=\kappa_6[\ce{V}]+\kappa_8[\ce{X}],\label{C-eq}\tag{C-eq}\\
\kappa_1[\ce{A}][\ce{E}]+\kappa_4[\ce{B}][\ce{E}]&=(\kappa_2+\kappa_3)[\ce{U}]
       +(\kappa_5+\kappa_6)[\ce{V}],\label{K-eq}\tag{E-eq}\\
\kappa_7[\ce{C}][\ce{F}]+\kappa_{10}[\ce{B}][\ce{F}]&=(\kappa_{11}+\kappa_{12})[\ce{Y}]
       +(\kappa_8+\kappa_9)[\ce{X}],\label{F-eq}\tag{F-eq}\\
(\kappa_2+\kappa_3)[\ce{U}]&=\kappa_1[\ce{A}][\ce{E}],\label{U-eq}\tag{U-eq}\\
(\kappa_5+\kappa_6)[\ce{V}]&=\kappa_4[\ce{B}][\ce{E}],\label{V-eq}\tag{V-eq}\\
(\kappa_8+\kappa_9)[\ce{X}]&=\kappa_7[\ce{C}][\ce{F}],\label{X-eq}\tag{X-eq}\\
(\kappa_{11}+\kappa_{12})[\ce{Y}]&=\kappa_{10}[\ce{B}][\ce{F}].\label{Y-eq}\tag{Y-eq}
\end{align}

\begin{prop}[Mass action constraints]\label{lem:maconstaints}
    The symbolic Jacobian $G$ can be realized as the Jacobian matrix of the mass action system \eqref{doubleeqma}, evaluated at a positive steady state, if and only if the parameters satisfy the following two constraints:
        \begin{equation}\label{eq:MA1}\tag{MA1} 
\dfrac{u_3}{u_2+u_3}\frac{b_4}{b_{10}}=\dfrac{y_{12}}{y_{11}+y_{12}}\dfrac{e_4}{e_1},
\end{equation}
\begin{equation}\tag{MA2}\label{eq:MA2}
\dfrac{v_6}{v_5+v_6}\dfrac{b_4}{b_{10}}=\dfrac{x_9}{x_8+x_9}\dfrac{f_7}{f_{10}}.             \end{equation}        
\end{prop}

\begin{remark}[Sufficiency of Prop.~\ref{lem:maconstaints}]
    Note that, for the purposes of this paper of excluding a bifurcation behavior, only the necessity of \eqref{eq:MA1} and \eqref{eq:MA2} is relevant.
 \end{remark}

\proof
Firstly, we record the partial derivatives at steady state, under mass action:
\begin{equation}\label{eq:pdma}
\begin{alignedat}{4}
a      &= \kappa_1[\ce{E}]  &\qquad
b_4    &= \kappa_4[\ce{E}]  &\qquad
b_{10} &= \kappa_{10}[\ce{F}] &\qquad
c      &= \kappa_7[\ce{F}], \\
e_1    &= \kappa_1[\ce{A}]  &
e_4    &= \kappa_4[\ce{B}]  &
f_7    &= \kappa_7[\ce{C}]  &
f_{10} &= \kappa_{10}[\ce{B}], \\
u_2    &= \kappa_2 &
u_3    &= \kappa_3 &
v_5    &= \kappa_5 &
v_6    &= \kappa_6, \\
x_8    &= \kappa_8 &
x_9    &= \kappa_9 &
y_{11} &= \kappa_{11} &
y_{12} &= \kappa_{12}.
\end{alignedat}
\end{equation}
Necessity: the two constraints are one the $\sigma$-symmetric of the other, so it suffices to prove one. The equations \eqref{A-eq} and \eqref{U-eq} give
\[
    (\kappa_2+\kappa_3)[\ce{U}]=\kappa_2[\ce{U}]+\kappa_{12}[\ce{Y}],
\]
and thus $\kappa_3[\ce{U}]=\kappa_{12}[\ce{Y}]$, i.e. 
\[u_3[\ce{U}]=y_{12}[\ce{Y}].\]
Note that this is exactly the steady state equation for $Q_1$ introduced in \eqref{Qchange}. On the other hand, equations \eqref{U-eq} and \eqref{Y-eq} give
\[
[\ce{U}]=\dfrac{\kappa_1[\ce{A}][\ce{E}]}{\kappa_2+\kappa_3}\qquad\text{and}\qquad
[\ce{Y}]=\dfrac{\kappa_{10}[\ce{B}][\ce{F}]}{\kappa_{11}+\kappa_{12}}.
\]
Jointly, and with the correct substitution, they yield:
\[
u_3 \dfrac{\kappa_1[\ce{A}][\ce{E}]}{\kappa_2+\kappa_3}
=y_{12} \dfrac{\kappa_{10}[\ce{B}][\ce{F}]}{\kappa_{11}+\kappa_{12}},
\]
that is,
\[
u_3 \dfrac{e_1[\ce{E}]}{u_2+u_3}
=y_{12} \dfrac{b_{10}[\ce{B}]}{y_{11}+y_{12}}.
\]
Finally, note that
\[
\dfrac{[\ce{B}]}{[\ce{E}]}
=\frac{\kappa_4[\ce{B}]}{\kappa_4[\ce{E}]}
=\dfrac{e_4}{b_4},
\]
and thus \eqref{eq:MA1} follows. Analogously, \eqref{eq:MA2} follows by performing symmetrically the operations on equations \eqref{C-eq}, \eqref{X-eq}, and \eqref{V-eq}.

Sufficiency: For any choice of the 16 parameters $\mathbf{p}$
 that satisfy
\eqref{eq:MA1} and \eqref{eq:MA2}, we construct a mass action system
which realizes such Jacobian configuration at a positive steady state. Informally,
the logic is as follows. The rates of the monomolecular reactions are
fixed a priori. Then, we may fix $[\overline{\ce{B}}]=1$, arbitrarily, due to rescale
freedom: even if arbitrary, I choose $\ce{B}$ because $\sigma(\ce{B})=\ce{B}$, so that it serves better as an organizing center. $[\overline{\ce{B}}]=1$ in turn fixes $\kappa_4=e_4$ and $\kappa_{10}=f_{10}$, as well as the
steady state concentrations $[\overline{\ce{E}}]=b_4/e_4$ and $[\overline{\ce{F}}]=b_{10}/f_{10}$. The
construction then proceeds by cascade via $\kappa_1=ae_4/b_4$ and
$[\overline{\ce{A}}]=b_4e_1/ae_4$, and so on. One can verify that the constructed
rate constants are as follows:
\begin{equation}
\begin{array}{cccccc}
\kappa_1=\dfrac{ae_4}{b_4}\;
&
\kappa_2=u_2\;
&
\kappa_3=u_3\;
&
\kappa_4=e_4\;
&
\kappa_5=v_5\;
&
\kappa_6=v_6\\[4mm]

\kappa_7=\dfrac{cf_{10}}{b_{10}}\;
&
\kappa_8=x_8\;
&
\kappa_9=x_9\;
&
\kappa_{10}=f_{10}\;
&
\kappa_{11}=y_{11}\;
&
\kappa_{12}=y_{12}
\end{array}
\end{equation}
In turn, the positive steady-state concentrations are
\begin{equation*}
\begin{alignedat}{3}
[\overline{\ce{A}}]
    &= \frac{b_4e_1}{ae_4}
&\qquad
[\overline{\ce{B}}]
    &= 1
&\qquad
[\overline{\ce{C}}]
    &= \frac{b_{10}f_7}{cf_{10}},
\\
[\overline{\ce{E}}]
    &= \frac{b_4}{e_4}
&
[\overline{\ce{F}}]
    &= \frac{b_{10}}{f_{10}}
&
[\overline{\ce{U}}]
    &= \frac{b_4e_1}{e_4(u_2+u_3)},
\\
[\overline{\ce{V}}]
    &= \frac{b_4}{v_5+v_6}
&
[\overline{\ce{X}}]
    &= \frac{b_{10}f_7}{f_{10}(x_8+x_9)}
&
[\overline{\ce{Y}}]
    &= \frac{b_{10}}{y_{11}+y_{12}}.
\end{alignedat}
\end{equation*}
It is straightforward to check that they satisfy \eqref{eq:pdma} and
that, under \eqref{eq:MA1} and \eqref{eq:MA2}, they satisfy the
steady-state equations. Only \eqref{A-eq}, \eqref{B-eq}, \eqref{C-eq}, which relate to the three substrates $\ce{A},\ce{B},\ce{C}$ need to be enforced via \eqref{eq:MA1} and \eqref{eq:MA2}. All the other equations are identically satisfied, and we omit listing them here. The equations  \eqref{A-eq}, \eqref{B-eq}, \eqref{C-eq} read, after simplifications:
\begin{align*}
\frac{b_4e_1}{e_4}
&=
u_2\frac{b_4e_1}{e_4(u_2+u_3)}
+y_{12}\frac{b_{10}}{y_{11}+y_{12}},
\label{A-eq'}\tag{A-eq'}\\
b_{10}+b_4
&=
u_3\frac{b_4e_1}{e_4(u_2+u_3)}
+y_{11}\frac{b_{10}}{y_{11}+y_{12}}
+v_5\frac{b_4}{v_5+v_6}
+x_9\frac{b_{10}f_{7}}{f_{10}(x_8+x_9)},
\label{B-eq'}\tag{B-eq'}\\
\frac{b_{10}f_{7}}{f_{10}}
&=
v_6\frac{b_4}{v_5+v_6}
+x_8\frac{b_{10}f_{7}}{f_{10}(x_8+x_9)},
\label{C-eq'}\tag{C-eq'}\\
\end{align*}
Due to the $\mathbb{Z}_2$ symmetry $\sigma$, enforcing \eqref{eq:MA1} solves \eqref{A-eq'} identically as enforcing \eqref{eq:MA2} solves \eqref{C-eq'}. We exemplify only for \eqref{A-eq'}:
\begin{equation*}
\begin{split}
u_2\frac{b_4e_1}{e_4(u_2+u_3)}
+y_{12}\frac{b_{10}}{y_{11}+y_{12}}
&=
\frac{b_4e_1u_2(y_{11}+y_{12})+b_{10}e_4y_{12}(u_2+u_3)}
     {e_4(u_2+u_3)(y_{11}+y_{12})}
\\
&=
\frac{b_4e_1u_2(y_{11}+y_{12})+b_4e_1u_3(y_{11}+y_{12})}
     {e_4(u_2+u_3)(y_{11}+y_{12})}
=\frac{b_4e_1}{e_4}.
\end{split}
\end{equation*}
For \eqref{B-eq'}, the simultaneous application of  \eqref{eq:MA1} and \eqref{eq:MA2} give
\begin{equation*}
\begin{split}
&u_3\frac{b_4e_1}{e_4(u_2+u_3)}
+y_{11}\frac{b_{10}}{y_{11}+y_{12}}
+v_5\frac{b_4}{v_5+v_6}
+x_9\frac{b_{10}f_{7}}{f_{10}(x_8+x_9)}
\\
&=
(y_{12}+y_{11})\frac{b_{10}}{y_{11}+y_{12}}
+(v_5+v_6)\frac{b_4}{v_5+v_6}
\\
&=b_{10}+b_4.
\end{split}
\end{equation*}
Thus all the steady-state equations are satisfied.
\endproof

\subsection{Steady-state variables}\label{sec:massv}

We introduce here a parametrization for the steady-state Jacobian under mass-action kinetics. We recall from Sec.~\ref{sec:preliminaries} the reduced Jacobian matrix $G_{\mathrm{red}}$ in the form of
\begin{equation*}
    G_{\mathrm{red}}=\begin{pmatrix}
        \mathbf{0}_2 & R\\
        L & -N
    \end{pmatrix},
\end{equation*}
and the diagonal matrix $D=\operatorname{diag}(\delta_1,1,\delta_3,\delta_4)$, with 
\begin{equation*}
    \delta_1:=\dfrac{e_4}{e_1},\qquad \delta_3:=\dfrac{b_4}{b_{10}},\qquad \delta_4:=\dfrac{b_4f_{10}}{b_{10}f_{7}}.
\end{equation*}
such that $DN$ is symmetric. In analogy, we introduce the four positive
ratios
\begin{equation}\label{eq:globalratios}
\xi_1:=\frac{u_3}{u_2},
\qquad
\xi_2:=\frac{y_{12}}{y_{11}},
\qquad
\eta_1:=\frac{x_9}{x_8},
\qquad
\eta_2:=\frac{v_6}{v_5}.
\end{equation}
In terms of these variables, the mass-action constraints \eqref{eq:MA1} and \eqref{eq:MA2} become
\begin{equation}\label{eq:manew}
\frac{\xi_1}{1+\xi_1}\frac{1}{\delta_1}
=
\frac{\xi_2}{1+\xi_2}\frac{1}{\delta_3},
\qquad
\frac{\eta_1}{1+\eta_1}\frac{1}{\delta_4}
=
\frac{\eta_2}{1+\eta_2}.
\end{equation}
Choosing $\delta_3>0$ as an unrestricted parameter, these constraints translate into
\begin{equation}\label{eq:globaldelta}
\delta_1
=
\delta_3
\frac{\xi_1(1+\xi_2)}
     {\xi_2(1+\xi_1)},
\qquad
\delta_4
=
\frac{\eta_1(1+\eta_2)}
     {\eta_2(1+\eta_1)}.
\end{equation}
Further natural variables are introduced:
\begin{equation}
\begin{aligned}
\alpha  &:= \delta_1 a,
&\qquad \tau_1   &:= \delta_1(u_2+u_3),
&\qquad \theta_1 &:= v_5+v_6, \\
\gamma &:= \delta_4 c,
&\qquad \tau_2   &:= \delta_3(y_{11}+y_{12}),
&\qquad \theta_2 &:= \delta_4(x_8+x_9),\\
& \qquad &\epsilon&:=\delta_3f_{10},
\end{aligned}
\end{equation}
 which - jointly with the untouched variables $e_4$ and $b_4$ - makes $DN$ and $DL$ read as follows:
 {\small
 \[DN= 
         \begin{pmatrix}
          \alpha+e_4+\tau_1 & e_4 & 0 & 0\\
            e_4 & b_4+e_4+\theta_1& b_4 &0\\
             0 & b_4 & b_4+\epsilon+\tau_2 & \epsilon\\
             0 & 0 & \epsilon &\epsilon+ \gamma+\theta_2
         \end{pmatrix} \qquad \text{and} \qquad DL=\begin{pmatrix}
    \alpha & 0\\
    -b_4 & -b_4\\
    -b_4 & -b_4\\
    0 & \gamma
\end{pmatrix}.\]}
\hspace{-0.25cm}
Moreover, since in this new variables \eqref{eq:manew} implies
\[
v_6=\frac{\eta_2}{1+\eta_2}\theta_1,
\qquad
x_9=\frac{\eta_2}{1+\eta_2}\theta_2,
\qquad
y_{12}=\frac{\xi_2}{\delta_3(1+\xi_2)}\tau_2,
\qquad
u_3=\frac{\xi_2}{\delta_3(1+\xi_2)}\tau_1,
\]
we have
\[
R=
\begin{pmatrix}
-\dfrac{\xi_2}{\delta_3(1+\xi_2)}\tau_1
&
0
&
\dfrac{\xi_2}{\delta_3(1+\xi_2)}\tau_2
&
0
\\[3mm]
0
&
\dfrac{\eta_2}{1+\eta_2}\theta_1
&
0
&
-\dfrac{\eta_2}{1+\eta_2}\theta_2
\end{pmatrix}.
\]
In conclusion, the 14 independent positive \textbf{mass-action steady-state variables} are:
\begin{equation}\label{massv}
\pmb{\mu}:=(\xi_1,\xi_2,\eta_1,\eta_2,
\delta_3,\alpha,\gamma,\epsilon,
\tau_1,\tau_2,\theta_1,\theta_2,e_4,b_4).
\end{equation}
Those variables globally describe the steady-state mass-action Jacobian space. Indeed, the inverse formulas through \eqref{eq:globaldelta} read ($e_4$ and $b_4$ remain unchanged):
\begin{equation*}
\begin{aligned}
a&=\frac{\alpha}{\delta_1},
&\qquad e_1&=\frac{e_4}{\delta_1},
&\qquad b_{10}&=\frac{b_4}{\delta_3},
&\qquad f_{10}&=\frac{\epsilon}{\delta_3},\\
c&=\frac{\gamma}{\delta_4},
& f_7&=\frac{\epsilon}{\delta_4},
& u_2&=\frac{\tau_1}{\delta_1(1+\xi_1)},
& u_3&=\frac{\xi_1\tau_1}{\delta_1(1+\xi_1)},\\
y_{11}&=\frac{\tau_2}{\delta_3(1+\xi_2)},
& y_{12}&=\frac{\xi_2\tau_2}{\delta_3(1+\xi_2)},
& v_5&=\frac{\theta_1}{1+\eta_2},
& v_6&=\frac{\eta_2\theta_1}{1+\eta_2},\\
x_8&=\frac{\theta_2}{\delta_4(1+\eta_1)},
& x_9&=\frac{\eta_1\theta_2}{\delta_4(1+\eta_1)},
\end{aligned}
\end{equation*}
and show that the change of variables is an analytic bijection between $\mathbf{p}=\mathbb{R}^{16}_{>0}$ under constraints \eqref{eq:MA1} and \eqref{eq:MA2} and $\pmb{\mu}=\mathbb{R}^{14}_{>0}$.

\subsection{A positivity certificate for $\mathcal{T}\ge0$}\label{sec:positivitycertificate}

We start with a polynomialization Lemma, where we clear the expression $\mathcal{T}$ in the variables \eqref{massv} from appearing denominators.

\begin{lemma}[Polynomialization Lemma]\label{lem:hom}
Let $\widetilde{\mathcal{T}}$ indicate the quantity $\mathcal{T}$, \eqref{eq:star}, expressed in the mass-action steady-state variables $\pmb{\mu}$, \eqref{massv}. It holds that 
\begin{equation}\label{eq:Gdefinition}
\widetilde{\mathcal{T}}^* 
:=
\delta_3^{12}\xi_1^5\eta_1^5
(1+\xi_2)^9(1+\eta_2)^9
\widetilde{\mathcal T}
\end{equation}
is a polynomial in the variables $\pmb{\mu}$.
\end{lemma}

\begin{proof}
    We recall that $\delta_1\delta_3\delta_4=\det D$, and the equality \eqref{eq:qep}, namely
    \begin{equation*}
        \det D \; g_{\mathrm{red}}(\lambda)\;\lambda^2 = \det(D\lambda^2+DN\lambda-DLR).
\end{equation*}        
By expanding both sides we have that
\begin{equation*}
\begin{split}
\delta_1\delta_3\delta_4\; (\mathfrak{c}_0\lambda^6+\mathfrak{c}_1 \lambda^5+\mathfrak{c}_2\lambda^4+\mathfrak{c}_3\lambda^3+\mathfrak{c}_4\lambda^2+\mathfrak{c}_5\lambda+\mathfrak{c}_6)\;\lambda^2=\\
(\mathfrak{a}_0\lambda^8+\mathfrak{a}_1 \lambda^7+\mathfrak{a}_2\lambda^6+\mathfrak{a}_3\lambda^5+\mathfrak{a}_4\lambda^4+\mathfrak{a}_5\lambda^3+\mathfrak{a}_6\lambda^2+\mathfrak{a}_7\lambda+\mathfrak{a}_8),
\end{split}
\end{equation*}
from which - by explicit comparison of coefficients we get that $\mathfrak{a}_8=\mathfrak{a}_7=0$ and that \begin{equation}\label{eq:ciai}
       \delta_1\delta_3\delta_4 \mathfrak{c}_i=\mathfrak{a}_i\qquad\qquad i=0,...,6. 
    \end{equation}

Now we note that $\Delta_5$ is homogeneous of degree 5 in the $\mathfrak{c}_i$ coefficients, and $\mathfrak{c}_6 \mathcal{T}$ is obtained from $P_6 \Delta_5$ in \eqref{eq:TfromDelta5} by subtraction of homogeneous terms of degree 6, so that $\mathfrak{c}_6\mathcal{T}$ is homogeneous of degree 6. Since $\mathfrak{c}_6=P_6-N_6$ is homogeneous of degree 1 then $\mathcal{T}$ is homogeneous of degree 5. Hence, using notation $\widetilde{\mathcal{T}}(\mathbf{c})$ for the symbolic expression built from the $\mathfrak{c}_i$ coefficients and $\widetilde{\mathcal{T}}(\mathbf{a})$ for the same expression but built from the $\mathfrak{a}_i$ coefficients, \eqref{eq:ciai} yields
\begin{equation}\label{eq:homodelta}
\widetilde{\mathcal{T}}(\mathbf{a})=\widetilde{\mathcal{T}}(\delta_1\delta_3\delta_4\mathbf{c})=(\delta_1\delta_3\delta_4)^5\widetilde{\mathcal{T}}.
\end{equation}
In turn, $\tilde{\mathcal{T}}(\mathbf{a})$ is itself a rational function and an explicit computation shows that the correct denominator clearance is by multiplication with a factor $\delta_3^2
\xi_2^4\eta_2^4
(1+\xi_1)^5(1+\eta_1)^5
(1+\xi_2)^4(1+\eta_2)^4$. This is performed in the attached \texttt{MATLAB} script \texttt{dualfutile\_ma\_nohopf.m}. In particular,
\[\delta_3^2
\xi_2^5\eta_2^5
(1+\xi_1)^5(1+\eta_1)^5
(1+\xi_2)^4(1+\eta_2)^4 \tilde{\mathcal{T}}(\mathbf{a})\]
is a polynomial in the mass action variables $\pmb{\mu}$. Note that the exponents of $\xi_2$ and $\eta_2$ are chosen to be 5, and not the minimal 4: this is because substituting \eqref{eq:globaldelta} provides denominators of degree $5$ in  $\xi_2,\eta_2$. We compute indeed:
\begin{equation}\label{eq:518}
\begin{split}
&\delta_3^2
\xi_2^5\eta_2^5
(1+\xi_1)^5(1+\eta_1)^5
(1+\xi_2)^4(1+\eta_2)^4 \tilde{\mathcal{T}}(\mathbf{a})\\=&\delta_3^2
\xi_2^5\eta_2^5
(1+\xi_1)^5(1+\eta_1)^5
(1+\xi_2)^4(1+\eta_2)^4(\delta_1\delta_3\delta_4)^5\widetilde{\mathcal{T}}\\
=&\delta_3^2
\xi_2^5\eta_2^5
(1+\xi_1)^5(1+\eta_1)^5
(1+\xi_2)^4(1+\eta_2)^4\bigg(\dfrac{\xi_1(1+\xi_2)}{\xi_2(1+\xi_1)}\delta_3^2\dfrac{\eta_1(1+\eta_2)}
{\eta_2(1+\eta_1)}\bigg)^5\widetilde{\mathcal{T}}\\
=&\delta^{12}_3 \xi_1^5\eta_1^5(1+\xi_2)^9(1+\eta_2)^9\tilde{\mathcal{T}}\quad=:\tilde{\mathcal{T}}^*,
\end{split}
\end{equation}
which is a polynomial in the variables $\pmb{\mu}$.
\end{proof}
Trivially, $\delta^{12}_3 \xi_1^5\eta_1^5(1+\xi_2)^9(1+\eta_2)^9>0$ yields 
\begin{equation*}
    \mathcal{T}\ge0\qquad \Longleftrightarrow\qquad  \tilde{\mathcal{T}}^*\ge0,
\end{equation*}
and thus Lemma \ref{lem:hom} justifies seeking a positivity certificate for $\widetilde{\mathcal{T}}^*\ge0$, which is what we present next.

\begin{prop}[Positivity of $\widetilde{\mathcal{T}}^*$]\label{prop:poscertificate} For the polynomial $\tilde{\mathcal{T}}^*$ in the variables \eqref{massv} it holds
\[\tilde{\mathcal{T}}^*>0.\]
Consequently, $\mathcal{T}>0$, as defined in \eqref{eq:star}.
\end{prop}
\begin{proof}
By exact computer-assisted polynomial expansion and coefficient collection, $\tilde{\mathcal{T}}^*>0$ can be written as
\begin{equation*}
\widetilde{\mathcal{T}}^*=\Sigma^+ +\Gamma_1 \mathcal{E} +\Gamma_2  \mathcal{E}^\sigma,
\end{equation*}
where $\Sigma^+>0$ is a polynomial in positive variables with only positive coefficients (verified in the attached \texttt{MATLAB} implementation), and $\Gamma_1,\Gamma_2>0$ since their explicit form is
\[
\begin{cases}
 \Gamma_1=\delta_3^4\alpha\tau_1\gamma\theta_2e_4^4\epsilon^3\xi_2^6\eta_1\eta_2^5
\Big[
(1+\xi_2)^3(1+\eta_2)^4
+\xi_2^3\eta_2^3
\bigl(
\eta_1(1+\eta_2)+\xi_1\eta_2
\bigr)
\Big]>0\\
\\
\Gamma_2=\delta_3^5\alpha\tau_1\gamma\theta_2e_4^3\epsilon^4\xi_1\xi_2^5\eta_2^6
\Big[
(1+\xi_2)^4(1+\eta_2)^3
+\xi_2^3\eta_2^3
\bigl(
\xi_1(1+\xi_2)+\xi_2\eta_1
\bigr)
\Big]>0
\end{cases}
\]
Thus, positivity of $\tilde{\mathcal{T}}^*$ reduces to prove positivity of $\mathcal{E}$ and $\mathcal{E}^\sigma$. Since - as we will specify below - they are related by the symmetry $\sigma$, it suffices proving positivity of $\mathcal{E}$. Explicitly, $\mathcal{E}$ reads as follows. \[
\begin{aligned}
\mathcal{E}=
\Bigg[
\
&e_4\Bigg(
2(\tau_2+\gamma+\theta_2)^3
+4\tau_2(\gamma+\theta_2)(\tau_2+\gamma+\theta_2)
\\
&\qquad
+(\alpha+\tau_1+\theta_1)\tau_2^2
+6(\alpha+\tau_1+\theta_1)^2
   (2\tau_2+\gamma+\theta_2)
\\
&\qquad
-(\alpha+\tau_1+\theta_1)
\bigl(
\gamma^2+\theta_2^2
+2\tau_2\gamma+2\tau_2\theta_2+2\gamma\theta_2
\bigr)
\Bigg)
\\[1mm]
&+\epsilon\Bigg(
\bigl(9(\alpha+\tau_1)+16\theta_1\bigr)
(\tau_2+\gamma+\theta_2)^2
\\
&\qquad
+\theta_1^2(\tau_2+\gamma+\theta_2)
+4(\alpha+\tau_1)^3
+14(\alpha+\tau_1)^2\theta_1
\\
&\qquad
+12(\alpha+\tau_1)\theta_1^2
+2\theta_1^3
\\
&\qquad
-\bigl(
3(\alpha+\tau_1)^2
+4(\alpha+\tau_1)\theta_1
\bigr)
(\tau_2+\gamma+\theta_2)
\Bigg)
\\[1mm]
&+e_4\epsilon
\Bigg(
2(\alpha+\tau_1+\theta_1)^2
+2(\tau_2+\gamma+\theta_2)^2
\\
&\hspace{42mm}
-2(\alpha+\tau_1+\theta_1)
 (\tau_2+\gamma+\theta_2)
\Bigg)
\Bigg].
\end{aligned}
\tag{$\mathcal{E}$}
\label{C1}
\]
There are only $39$ monomials with negative coefficient, which are exactly the expansion of the three
negative products $-(...)(...)$ above, which contain respectively $15+15+9$ monomials. We rewrite 
\begin{equation*}
\mathcal{E}=\bigg(e_4 \Sigma^{(e_4)} +\epsilon \Sigma^{(\epsilon)}+e_4\epsilon \Sigma^{(e_4\epsilon)}\bigg).
\end{equation*}
Clearly, nonnegativity of all $\Sigma^{(e_4)}$, $\Sigma^{(\epsilon)}$, and $\Sigma^{(e_4\epsilon)}$ implies nonnegativity of $\mathcal{E}$. This is what we show now.\\

\emph{Nonnegativity of $\Sigma^{(e_4)}$.} We expand as a follows:\\
\begin{equation*}\begin{split}\Sigma^{(e_4)}=&(\tau_2+\gamma+\theta_2)[
2(\tau_2+\gamma+\theta_2)^2-(\alpha+\tau_1+\theta_1)(\tau_2+\gamma+\theta_2)+6(\alpha+\tau_1+\theta_1)^2]\\&+
4\tau_2(\gamma+\theta_2)(\tau_2+\gamma+\theta_2)+6(\alpha+\tau_1+\theta_1)^2\tau_2+2(\alpha+\tau_1+\theta_1)\tau_2^2\\
>&(\tau_2+\gamma+\theta_2)[
2(\tau_2+\gamma+\theta_2)^2-(\alpha+\tau_1+\theta_1)(\tau_2+\gamma+\theta_2)+6(\alpha+\tau_1+\theta_1)^2]\\
=&(\tau_2+\gamma+\theta_2)\frac{1}{8}\bigg[[4(\tau_2+\gamma+\theta_2)-(\alpha+\tau_1+\theta_1)]^2+47(\alpha+\tau_1+\theta_1)^2\bigg]>0.\end{split}
\end{equation*}

\emph{Nonnegativity of $\Sigma^{(\epsilon)}$.}
We view $\Sigma^{(\epsilon)}$ as a quadratic polynomial in
$\tau_2+\gamma+\theta_2$: 
\[
\begin{aligned}
\Sigma^{(\epsilon)}
={}&
\bigl(9(\alpha+\tau_1)+16\theta_1\bigr)
(\tau_2+\gamma+\theta_2)^2
\\
&+
\bigl[
\theta_1^2
-4\theta_1(\alpha+\tau_1)
-3(\alpha+\tau_1)^2
\bigr]
(\tau_2+\gamma+\theta_2)
\\
&+
4(\alpha+\tau_1)^3
+14(\alpha+\tau_1)^2\theta_1
+12(\alpha+\tau_1)\theta_1^2
+2\theta_1^3.
\end{aligned}
\]
Its leading coefficient $9(\alpha+\tau_1)+16\theta_1$ is positive. Moreover, its discriminant is
\[
\begin{aligned}
&
\bigl[
\theta_1^2
-4\theta_1(\alpha+\tau_1)
-3(\alpha+\tau_1)^2
\bigr]^2
\\
&\quad
-4\bigl(9(\alpha+\tau_1)+16\theta_1\bigr)
\Bigl[
4(\alpha+\tau_1)^3
+14(\alpha+\tau_1)^2\theta_1
+12(\alpha+\tau_1)\theta_1^2
+2\theta_1^3
\Bigr]
\\
&=
-\Bigl[
135(\alpha+\tau_1)^4
+736\theta_1(\alpha+\tau_1)^3
+1318\theta_1^2(\alpha+\tau_1)^2
\\
&\hspace{25mm}
+848\theta_1^3(\alpha+\tau_1)
+127\theta_1^4
\Bigr]
<0.
\end{aligned}
\]
Hence the quadratic has no real roots and, since its leading
coefficient is positive, $\Sigma^{(\epsilon)}>0.$

\emph{Nonnegativity of $\Sigma^{(e_4 \epsilon)}$.} We expand as a sum of squares:
$$\Sigma^{(e_4 \epsilon)}=(\alpha+\tau_1+\theta_1)^2
+(\tau_2+\gamma+\theta_2)^2+
[(\alpha+\tau_1+\theta_1)-
(\tau_2+\gamma+\theta_2)]^2>0.$$\\
$\mathcal{E}^\sigma
$ follows by noting that its expression is obtained from \eqref{C1} by $\sigma$-inherited bijection
\[
\alpha\leftrightarrow\gamma,\qquad
\tau_1\leftrightarrow\theta_2,\qquad
\theta_1\leftrightarrow\tau_2,\qquad
e_4\leftrightarrow\epsilon,
\]
which implies its positivity via identical arguments as above.
\end{proof}

\begin{remark}[Symmetry $\sigma$ in the mass-action variables $\pmb{\mu}$]
Our choice of steady-state variables \eqref{massv} does not fully respect the structural $\sigma$-symmetry in arbitrarily fixing $\delta_2=1$ and using $\delta_3$ as a free variable. This explains the slight apparent asymmetry between $\Gamma_1$
 and $\Gamma_2$, even though positivity of $\mathcal{E}$ corresponds via $\sigma$ to positivity of $\mathcal{E}^{\sigma}$.
 \end{remark}
 
\subsection{Hopf exclusion for mass action kinetics}\label{sec:exclusion}

To close the argument, we still need two technical steps. The first step is checking that the fourth Hurwitz determinant $\Delta_4$ is positive under mass action.
\begin{prop}\label{prop:delta4pos}
  Under the mass action constraints \eqref{eq:MA1} and \eqref{eq:MA2},
  \begin{equation}
      \Delta_4=\mathfrak{c}_4(\Delta_3-\mathfrak{c}_1\mathfrak{c}_5)
-\mathfrak{c}_5(\mathfrak{c}_2\Delta_2-2\mathfrak{c}_1\mathfrak{c}_4)
-\mathfrak{c}_5^2+\mathfrak{c}_1\Delta_2\mathfrak{c}_6>0.
  \end{equation}
\end{prop}
\proof

Explicit computation of $\Delta_4$ shows that $\Delta_4$ contains 1,712,913 monomials (of which 266 negative). Crucially, all negative monomials share the common factor $ace_4f_{10}v_6y_{12}$, which is the one associated to the unstable core fixed by the $\mathbb{Z}_2$ symmetry $\sigma$, see Sec.~\ref{sec:sufficient}. We can then collect positive and negative monomials as follows:
\begin{equation}
    \Delta_4=P_{\Delta_4}-N_{\Delta_4}=P_{\Delta_4}-ace_4f_{10}v_6y_{12}\;\mathcal{N}_{\Delta_4}.
\end{equation}
 Thus, since \eqref{eq:MA1} and \eqref{eq:MA2} imply 
\begin{equation}
    e_4f_{10}v_6y_{12}(u_2+u_3)(x_8+x_9)=e_1f_{7}x_9u_3(y_{11}+y_{12})(v_5+v_6),
\end{equation}
we write
\begin{equation*}
\begin{split}
   (u_2+u_3)(x_8+x_9) \Delta_4&=(u_2+u_3)(x_8+x_9)P_{\Delta_4}-(u_2+u_3)(x_8+x_9)ace_4f_{10}v_6y_{12}\;\mathcal{N}_{\Delta_4}\\
   &=(u_2+u_3)(x_8+x_9)P_{\Delta_4}-(y_{11}+y_{12})(v_5+v_6)acf_{7}e_1x_9u_3\;\mathcal{N}_{\Delta_4}.
   \end{split}
\end{equation*}
An explicit expansion of the right-hand-side shows that it is coefficientwise positive, with 4,295,079 terms, thus $\Delta_4>0$. See the attached file \texttt{dualfutile\_positivity\_checks.m}.
\endproof

The second step simply shows that a nonzero imaginary pair requires $\mathfrak{c}_6>0$.
\begin{lemma}\label{lem:mac6pos}
    Under the mass action constraints \eqref{eq:MA1} and \eqref{eq:MA2}, a Jacobian with nonzero purely imaginary eigenvalues implies $\mathfrak{c}_6>0$.
\end{lemma}
\proof
For the system under  mass action constraints, we recall that the Hurwitz determinants $\Delta_2,\Delta_3>0$ by Prop.~\ref{prop:pbb}, and $\Delta_4>0$ by Prop.~\ref{prop:delta4pos}. Therefore, assume $\lambda=i\omega$ is a root of the characteristic polynomial \[g_{\mathrm{red}}(\lambda)=\lambda^6+\mathfrak{c}_1 \lambda^5+ \mathfrak{c}_2 \lambda^4+\mathfrak{c}_3\lambda^3+\mathfrak{c}_4\lambda^2+\mathfrak{c}_5\lambda+\mathfrak{c}_6.\]This, as recalled above, forces automatically $\Delta_5=0$. We get that
\begin{equation}
-\omega^6+\mathfrak{c}_1 i\omega^5+ \mathfrak{c}_2 \omega^4-\mathfrak{c}_3i\omega^3-\mathfrak{c}_4\omega^2+\mathfrak{c}_5i\omega+\mathfrak{c}_6=0.
\end{equation}
Splitting real from imaginary part yields:
\begin{equation}
    \begin{cases}
        -\omega^6+ \mathfrak{c}_2 \omega^4-\mathfrak{c}_4\omega^2+\mathfrak{c}_6=0\\
        \mathfrak{c}_1 \omega^5-\mathfrak{c}_3\omega^3+\mathfrak{c}_5\omega=0
    \end{cases}
\end{equation}
Since $\omega>0$, we may  divide the second equality by $\omega$, and using notation $\Omega=\omega^2$ we get
\begin{equation}
    \begin{cases}
          -\Omega^3+ \mathfrak{c}_2 \Omega^2-\mathfrak{c}_4\Omega+\mathfrak{c}_6=0\\
        \mathfrak{c}_1 \Omega^2-\mathfrak{c}_3\Omega+\mathfrak{c}_5=0
    \end{cases}
\end{equation}
To express the above equalities in terms of $\Delta_2=\mathfrak{c}_1\mathfrak{c}_2-\mathfrak{c}_3$, we multiply the first equality by $\mathfrak{c}_1$ and add the second equality multiplied by $\Omega$, which yields
\begin{equation}
\begin{split}
    0&=\mathfrak{c}_1(-\Omega^3+ \mathfrak{c}_2 \Omega^2-\mathfrak{c}_4\Omega+\mathfrak{c}_6)+\Omega( \mathfrak{c}_1 \Omega^2-\mathfrak{c}_3\Omega+\mathfrak{c}_5)\\
    &=-\mathfrak{c}_1\Omega^3+ \mathfrak{c}_1\mathfrak{c}_2 \Omega^2-\mathfrak{c}_1\mathfrak{c}_4\Omega+\mathfrak{c}_1\mathfrak{c}_6+\mathfrak{c}_1\Omega^3-\mathfrak{c}_3\Omega^2+\mathfrak{c}_5\Omega\\
    &=\Delta_2\Omega^2+(\mathfrak{c}_5-\mathfrak{c}_1\mathfrak{c}_4)\Omega+\mathfrak{c}_1\mathfrak{c}_6
\end{split}
\end{equation}
Now, similarly, we want to express the above also in terms of $\Delta_3=\mathfrak{c}_3\Delta_2+\mathfrak{c}_1\mathfrak{c}_5-\mathfrak{c}_1^2\mathfrak{c}_4$ so to be able to leverage the formula for $\Delta_5$ directly. To do so, we multiply $\Delta_2\Omega^2+(\mathfrak{c}_5-\mathfrak{c}_1\mathfrak{c}_4)\Omega+\mathfrak{c}_1\mathfrak{c}_6=0$ by $\mathfrak{c}_1$ and use the imaginary-equation relation $\mathfrak{c}_1 \Omega^2=\mathfrak{c}_3\Omega-\mathfrak{c}_5$, which yields:
\begin{equation}
\begin{split}
    0&=\mathfrak{c}_1\Delta_2\Omega^2+\mathfrak{c}_1(\mathfrak{c}_5-\mathfrak{c}_1\mathfrak{c}_4)\Omega+\mathfrak{c}_1^2\mathfrak{c}_6\\
    &=\Delta_2(\mathfrak{c}_3\Omega-\mathfrak{c}_5)+\mathfrak{c}_1(\mathfrak{c}_5-\mathfrak{c}_1\mathfrak{c}_4)\Omega+\mathfrak{c}_1^2\mathfrak{c}_6\\
    &=(\mathfrak{c}_3\Delta_2+\mathfrak{c}_1\mathfrak{c}_5-\mathfrak{c}_1^2\mathfrak{c}_4)\Omega -\mathfrak{c}_5\Delta_2+\mathfrak{c}_1^2\mathfrak{c}_6\\
    &=\Delta_3\Omega-\mathfrak{c}_5\Delta_2+\mathfrak{c}_1^2\mathfrak{c}_6
\end{split}
\end{equation}
Thus, $\Delta_3\Omega=\mathfrak{c}_5\Delta_2-\mathfrak{c}_1^2\mathfrak{c}_6$ and \eqref{HurIddelta25} 
implies that, at a purely imaginary root,
\[0=\Delta_2\Delta_5=\Delta_3\Delta_4\Omega-\mathfrak{c}_6\Delta_3^2.\]
Strict positivity of $\Delta_2$, $\Delta_3$, and $\Delta_4$ implies $\mathfrak{c}_6 >0$. 
\endproof

Gluing together Lemma \ref{lem:mac6pos}, Lemma \ref{lem:centralhurwitz}, and the positivity certificate Prop.~\ref{prop:poscertificate}, the main theorem follows:

\begin{theorem}\label{thm:main}
    Under the mass action constraints \eqref{eq:MA1} and \eqref{eq:MA2}, $G$ does not possess nonzero purely imaginary eigenvalues. In particular, the associated ODE \eqref{doubleeqma} system does not have the capacity for Hopf bifurcation at a positive steady state.
\end{theorem}

\section{Implementations}\label{sec:computations}

The proofs contain several claims whose verification requires computer algebra. To this end, four \texttt{MATLAB} scripts accompany the paper:
\begin{center}
\texttt{dualfutile\_positivity\_checks.m},\\
\texttt{dualfutile\_ma\_nohopf.m},\\
\texttt{dualfutile\_identity\_checks.m}, \\
\texttt{dualfutile\_pr\_hopf.m}.
\end{center}
These scripts are described below and are available at \begin{center}\url{https://github.com/nvassena/dualfutile_hopfnohopf} \; . \end{center}

\paragraph{\texttt{dualfutile\_positivity\_checks.m}.} This code implements all the minor positivity statements in the paper, i.e. with the exception of the ones in Prop.~\ref{prop:poscertificate}. The positivity statements are listed in the following table, accompanied by the location in the paper. All checks are done by explicit expansion and coefficient-wise positivity. Here an overview.
\begin{center}
\begin{tabular}{@{}p{0.75\textwidth}l@{}}
\toprule
\textbf{Target to be shown} & \textbf{Location}\\
\midrule
$\mathfrak{c}_i$ for $i=0,...,4$ & Prop.~\ref{prop:pbb}\\
$P_5,N_5,P_6,N_6$ & Sec.~\ref{sec:sufficient}\\
$\Delta_i$ for $i=1,...3$ & Prop.~\ref{prop:pbb}\\
$N_6 P_5-N_5 P_6$ & Prop.~\ref{prop:pbb}\\
$P_5 \Delta_2-\mathfrak{c}_1^2 P_6$ & Prop.~\ref{prop:pbb}\\    $\mathfrak{c}_1 \mathfrak{c}_2-2\mathfrak{c}_3$ & Prop.~\ref{prop:pbb}\\
$(u_2+u_3)(x_8+x_9)P_{\Delta_4}-(y_{11}+y_{12})(v_5+v_6)acf_{7}e_1x_9u_3\; \mathcal{N}_{\Delta_4}$ & Prop.~\ref{prop:delta4pos}\\
\bottomrule
\end{tabular}
\end{center}
The implementation records explicitly the number of the positive monomials in the expansions, as indicated in the text.

\paragraph{\texttt{dualfutile\_ma\_nohopf.m}.}
This is the main computer-algebra verification underlying Prop.~\ref{prop:poscertificate}.  The script starts from the quadratic eigenvalue problem on the RHS of \eqref{eq:qep}, and computes its characteristic coefficients $\mathbf{a}=(\mathfrak{a}_0,...,\mathfrak{a}_6)$, expressed already in the mass-action variables for faster computation. A small delicacy: at this stage, slightly different mass action variables are used, namely
\[\pmb{\mu}'=\delta_1,\delta_3,\delta_4,\alpha,\tau_1,\theta_1,
\tau_2,\gamma,\theta_2,\epsilon,e_4,b_4,
\frac{u_3}{\tau_1},\frac{v_6}{\theta_1}.\]
The advantage is that the matrix $R$ has polynomial entries in such variables:
\[R=
\begin{pmatrix}
-\dfrac{u_3}{\tau_1}\tau_1
&0
&\dfrac{u_3}{\tau_1}\tau_2
&0\\[2mm]
0
&\dfrac{v_6}{\theta_1}\theta_1
&0
&-\dfrac{v_6}{\theta_1}\theta_2
\end{pmatrix},
\]
where the mass-action constraints \eqref{eq:MA1} and \eqref{eq:MA2} have been used in the form:
\[
\frac{u_3}{\tau_1}=\frac{y_{12}}{\tau_2},
\qquad
\frac{v_6}{\theta_1}=\frac{x_9}{\theta_2}.
\]
The script then forms 
\[
\widetilde{\mathcal{T}}(\mathbf{a)}
=
P_5\Delta_4-P_6B_1^+-N_5B_2^+.
\]
The variables $\pmb{\mu}'$ are computationally very convenient for this first stage, but they are not all global variables since e.g. 
\[\dfrac{u_3}{\tau_1}=\dfrac{u_3}{\delta_1(u_2+u_3)}<\frac{1}{\delta_1}.\]
They are thus not suited for the positivity certificate itself. For this reason, the script subsequently perform the substitution 
\[\pmb{\mu}'\mapsto \pmb{\mu},\]
jointly with the exact denominator clearance \eqref{eq:518}
\[\tilde{\mathcal{T}}^*=\delta_3^2
\xi_2^5\eta_2^5
(1+\xi_1)^5(1+\eta_1)^5
(1+\xi_2)^4(1+\eta_2)^4 \tilde{\mathcal{T}}(\mathbf{a}).\]
%En passant, the script also verifies that the exponent vector $(2,5,5,5,5,4,4)$ is minimal in rendering the resulting expression a polynomial. 
The central step is the verification of the exact coefficient-wise positivity of $\Sigma^+$, i.e.
\[
\Sigma^+=\widetilde{\mathcal{T}}^{*}
-\Gamma_1\mathcal{E}-\Gamma_2\mathcal{E}^\sigma
\geq_{\mathrm{coeff}}0.
\]
Here, for computational simplicity, $\Gamma_1\mathcal{E}$ and $\Gamma_2\mathcal{E}^\sigma$ are constructed from scratch. The detailed positivity proofs for $\Sigma^{(e_4)},\Sigma^{(\epsilon)},\Sigma^{(e_4\epsilon)}$ in Prop.~\ref{prop:poscertificate} are further re-checked for safety. %For further safety, the script also verifies that the negative coefficients of $\widetilde{\mathcal{T}}^{*}$ corresponds exactly to those predicted by$\Gamma_1E$ and $\Gamma_2E^\sigma$ in Prop.~\ref{prop:poscertificate}. 
 %The script also reports summary statistics on the polynomial expansions andtheir coefficients. 
  All computations are performed exactly with integer
coefficients. Since \texttt{MATLAB} stores these integers in double precision,
the script verifies throughout that they remain within the range of exact
integer representation. It also checks that the encoding of the monomial
exponents does not overflow. On a MacBook Pro equipped with an Apple M2 chip and \texttt{MATLAB} R2026a, a complete run of the verification took 134.9 seconds.

\paragraph{\texttt{dualfutile\_identity\_checks.m}.} This file is only for sanity check and tests three equalities that have been proven analytically in the paper. It checks the following.
\begin{center}
\begin{tabular}{@{}p{0.57\textwidth}l@{}}
\toprule
\textbf{Target Identity} & \textbf{Location}\\
\midrule
$\det(\lambda \operatorname{Id}-G)=\lambda^3\det (\lambda \operatorname{Id}-G_{\mathrm{red}})$ & Sec.~\ref{sec:preliminaries}\\
$\Delta_2\Delta_5=(\mathfrak{c}_5\Delta_2-\mathfrak{c}_1^2\mathfrak{c}_6)\Delta_4-\mathfrak{c}_6\Delta_3^2$ & Prop.~\ref{prop:Hdelta25}\\
$P_6\Delta_5=(N_6P_5-N_5P_6)\Delta_4+\mathfrak{c}_6(P_5\Delta_4-P_6B_1^+)$ & Lemma~\ref{lem:centralhurwitz}\\
\bottomrule
\end{tabular}
\end{center}

\paragraph{\texttt{dualfutile\_pr\_hopf.m}.} This file simply checks the spectral computations for the Hopf example under parameter-rich kinetics in Sec.~\ref{sec:examplehopfpr}. It is included because the scale-separation of the variables may produce inconclusive results, depending on the approximation, if a too naive numerical approach is taken.  The script therefore already employs high-precision numerical eigenvalue computations.

\section{Discussion on AI-contribution}\label{sec:discussionai}

The mathematical backbone of this paper was developed over approximately 48 hours, from 7 to 9 August 2026, through extensive interaction with \texttt{GPT-5.6 Sol}, accessed through ChatGPT with a \textsc{Plus} subscription. For clarity of presentation, I group my interaction with the \texttt{LLM} in five logical steps, although steps 2--4 each consisted of dozens of individual queries. The initial exploration was conducted within a single conversation with \texttt{Very High} reasoning capability. Subsequent checks and counterchecks were aided also by different and independent conversations. Here below a representation of this interaction, schematically represented as input and output.

\begin{center}
\textbf{In-1}: Does $G$ admit purely imaginary eigenvalues for a choice of the 16 parameters?\\
\textbf{Out-1}: Yes, an example was produced.\\
$\downarrow$\\
\textbf{In-2}: Adapt this example to the mass action case.\\
\textbf{Out-2}: No mass action example was found.\\
$\downarrow$\\
\textbf{In-3}: Prove absence of Hopf bifurcation under mass action kinetics.\\
\textbf{Out-3}: Inconclusive results.\\
$\downarrow$\\
\textbf{In-4}: Work with $\mathfrak{c}_i=P_i-N_i$, $i=5,6$, prove Conjecture \ref{conjecture}.\\
\textbf{Out-4}: Several inconclusive results, including Lemma \ref{lem:centralhurwitz}.\\
$\downarrow$\\
\textbf{In-5}: Check $\mathcal{T}\ge0$ under mass action kinetics.\\
\textbf{Out-5}: Nonnegativity certificate for $\mathcal{T}\ge0$ produced.
\end{center}

The example of a Hopf point for $G$ under parameter-rich kinetics was found in very few queries: I have then slightly modified the example. Although inconclusive, Step 2 and Step 3 were very valuable explorations: the content of Sec.~\ref{sec:preliminaries}, Sec.~\ref{sec:massc} and Sec.~\ref{sec:massv}, Sec.~\ref{sec:exclusion} was essentially developed and learned here, including the formulation of the associated quadratic eigenvalue problem, and Lemma \ref{lem:c6pos}. I assess my central contribution to be in Step 4, where I turned again to the parameter-rich kinetics perspective, and suggested to work on Conjecture \ref{conjecture}, namely under the assumption $\mathfrak{c}_6>0$. In particular, my structural suggestion in splitting $\mathfrak{c}_6=P_6-N_6$ and $\mathfrak{c}_5=P_5-N_5$ and the request of expressing $\Delta_5$ in terms of $N_6P_5-N_5P_6 >0$ lead rather fast to several seemingly equivalent Lemmas of the kind of Lemma \ref{lem:centralhurwitz}. The work mostly concentrated in this step, where I pursued many different structural routes that I did not report here in detail. As all of these attempts ultimately failed, I have asked to test few of the sufficient results under mass action kinetics. The quantity $\mathcal{T}$ appeared to be the most suited for the purposes, according to preliminary tests, even though I cannot formally be sure that this is the best route. Still, the positivity certificate for $\mathcal{T}$ under mass action kinetics was provided in a single query after a time-reasoning of 61 minutes and 8 seconds. My independent check of the entire conclusions, as well as framing them in an acceptable form, took much more time than the joint development of the backbone itself. In particular, the originally proposed positivity certificate relied on a considerably more involved route via Bernstein polynomials representation \cite{BernCert:08, farouki:12bern}. This route included the introduction of further bounded variables, and the splitting of the parameter space into four distinct regions, which were analyzed one by one. Upon checking and writing up the details for the paper, I suspected a more effective underlying structure, which I asked the \texttt{LLM} to investigate. This produced the current certificate, which relies on the introduction of the global variables \eqref{massv} based on the ratios between partial derivatives of reaction rates with respect to the same reactants. The \texttt{MATLAB} codes to test the computer algebra were also written by the \texttt{LLM}, improved under my requests, and independently checked.

Finally, I report two informal and subjective observations from this experience. First, I found the \texttt{LLM} particularly effective at the rapid manipulation of symbolic expressions and the checking of large polynomials: this allows - in a very short time - the exploration of various paths that, at a first sight, would seem mathematically similar, so that I would probably not have spent so much time in testing many of them after a few initial failures. Second, this interaction appeared markedly path-dependent: repeated attempts within the same conversation tended to follow related routes unless a genuinely different direction was explicitly suggested. In particular, answers would typically not improve on the same question after a first failure till a human-led change of direction was introduced. Switching between parallel approaches proved particularly useful for this specific problem.

\section{Conclusion}\label{sec:conclusion}

This paper proves that the sequential and distributive dual futile cycle is able to support Hopf bifurcations under parameter-rich kinetics but not under mass action kinetics. The overall logic of the proof follows a Routh--Hurwitz scheme whose computational size, unfortunately, prevents a complete structural interpretation. Still, I would like to draw the attention on few structural features of relevance. 

First, a prominent role in the proof is played by the ratio variables:
\begin{equation*}
  \dfrac{e_4}{e_1},\qquad\dfrac{b_4}{b_{10}},\qquad \dfrac{f_{10}}{f_7},\qquad \dfrac{u_3}{u_2},
\qquad
\dfrac{y_{12}}{y_{11}},
\qquad
\dfrac{x_9}{x_8},
\qquad
\dfrac{v_6}{v_5}.
\end{equation*}
These are precisely all the possible ratios between partial derivatives of reaction rates with respect to the same reactant. In fact, the two derivatives forming each ratio occur in the same column of the symbolic Jacobian $G$, and the seven ratios above exhaust all such possibilities. This clean change of variables reflects the fact that seven species ($\ce{B}$,$\ce{E}$,$\ce{F}$,$\ce{U}$,$\ce{V}$,$\ce{X}$,$\ce{Y}$) participate as a reactant in exactly two reactions, while two species $\ce{A}$, and $\ce{C}$, which can be viewed as `first and last substrate in the chain' participate only in one. This is a specific structural feature of these phosphorylation networks.

Second, every negative monomial in the characteristic coefficients contains one of the three unstable-core expressions
\[
b_4cf_{10}v_6y_{12},\qquad ab_{10}e_4v_6y_{12},\qquad ace_4f_{10}v_6y_{12},
\]
all of which share the factor $v_6y_{12}$. Nevertheless, positivity of all the characteristic coefficients $\mathfrak{c}_i$ would not by itself preclude a Hopf bifurcation either, even though any negative summand of the fourth Hurwitz determinant $\Delta_4$ includes the factor  $ace_4f_{10}v_6y_{12}$. What really remains missing is an explanation of how the network structure generates negative contributions to the Hurwitz determinants. Their explicit expansions exhibit clear recurring patterns, but the structural origin of these patterns remains elusive. The explicit manipulations of the negative terms of $\mathcal{E}$ in the positivity certificate consistently revealed a quadratic structure, plausibly originating from underlying squared expressions. Despite repeated attempts, I have not found a higher-level formulation that explains this structure and avoids the computer-assisted certificate. Such an upgraded approach is possibly needed to tackle Conjecture \ref{conjecture}, which has so far resisted any attempt.

Third, the $\mathbb{Z}_2$ symmetry $\sigma$ has been used primarily to reduce certain computations by half.

Whether sequential and distributive $n$-futile cycles, for $n\ge3$, admit in general Hopf bifurcation under parameter-rich or mass-action kinetics is - to my knowledge - still an open question. Preliminary investigations for $n=3$ are very promising about this possibility. A natural starting point would be to assess whether any $n$-futile cycle admits such a clean formulation as a quadratic eigenvalue problem as for $n=2$. This appears natural, since such formulation only required introducing new variables $[\ce{A}]+[\ce{U}]$ and $[\ce{C}]+[\ce{X}]$, namely the first ($\ce{A}$) and the last ($\ce{C}$) substrate in the chain, summed to their respective associated intermediate, $\ce{U}$ and $\ce{X}$. A general structural description and characterization of this phenomenon w.r.t. $n$ is a natural direction for future work.

\bibliographystyle{amsalpha}
\bibliography{references}{}

@article{Blokhuis25,
  title={Stoichiometric recipes for periodic oscillations in reaction networks},
  author={Blokhuis, Alexander and Stadler, Peter F and Vassena, Nicola},
  journal={Proceedings of the Royal Society A},
  year={2026}
}

@book{GolubitskySymmBook,
  title={Singularities and Groups in Bifurcation Theory: Volume II},
  author={Golubitsky, Martin and Stewart, Ian and Schaeffer, David G},
  volume={69},
  year={2012},
  publisher={Springer Science \& Business Media}
}

@article{TorresFeliu21,
author = {Torres, Ang\'{e}lica and Feliu, Elisenda},
title = {Symbolic Proof of Bistability in Reaction Networks},
journal = {SIAM Journal on Applied Dynamical Systems},
volume = {20},
number = {1},
pages = {1-37},
year = {2021},
doi = {10.1137/20M1326672}
}

@article{ThomsonGunawardena:09,
  author =       {Thomson, Matthew and Gunawardena, Jeremy},
  doi =          {10.1038/nature08102},
  journal =      {Nature},
  number =       {7252},
  pages =        {274-277},
  title =        {Unlimited multistability in multisite
                  phosphorylation systems},
  volume =       {460},
  year =         {2009},
}

@article{AliciaetalToric:2012,
  title =        {Chemical reaction systems with toric steady states},
  author =       {P{\'e}rez Mill{\'a}n, Mercedes and Dickenstein,
                  Alicia and Shiu, Anne and Conradi, Carsten},
  journal =      {Bulletin of Mathematical Biology.},
  volume =       {74},
  pages =        {1027-1065},
  year =         {2012},
  doi =          {10.1007/s11538-011-9685-x},
}

@article{HellRendall:2015,
  title =        {A proof of bistability for the dual futile cycle},
  author =       {Hell, Juliette and Rendall, Alan D},
  journal =      {Nonlinear Analysis: Real World Applications},
  volume =       {24},
  pages =        {175-189},
  year =         {2015},
  doi =          {10.1016/j.nonrwa.2015.02.004},
}

@Article{Muller:12,
  author =       {M{\"u}ller, Stefan and Regensburger, Georg},
  title =        {Generalized mass action systems: Complex balancing
                  equilibria and sign vectors of the stoichiometric
                  and kinetic-order subspaces},
  journal =      {SIAM Journal on Applied Mathematics},
  volume =       {72},
  number =       {6},
  pages =        {1926-1947},
  year =         {2012},
  doi =          {10.1137/1108470},
}

@article{VasStad23,
author={Vassena, Nicola and Stadler, Peter F},
  title={Unstable Cores are the source of instability in chemical reaction networks},
  journal={Proceedings of the Royal Society A},
 volume={480},
  number={2285},
  pages={20230694},
  year={2024},
  publisher={The Royal Society},
  doi={10.1098/rspa.2023.0694}
}

@article{MM13,
  title={Die {K}inetik der {I}nvertinwirkung},
  author={Michaelis, L. and Menten, M. L.},
  journal={Biochem. Z.},
  volume={49},
  pages={333-369},
  year={1913},
}

@article{conradi2024distributive,
  title={In distributive phosphorylation catalytic constants enable non-trivial dynamics},
  author={Conradi, Carsten and Mincheva, Maya},
  journal={Journal of Mathematical Biology},
  volume={89},
  number={2},
  pages={20},
  year={2024},
  publisher={Springer}
}

@article{farouki:12bern,
  title={The {B}ernstein polynomial basis: A centennial retrospective},
  author={Farouki, Rida T},
  journal={Computer Aided Geometric Design},
  volume={29},
  number={6},
  pages={379--419},
  year={2012},
  publisher={Elsevier}
}

@article{BernCert:08,
  title={Certificates of positivity in the {B}ernstein basis},
  author={Boudaoud, Fatima and Caruso, Fabrizio and Roy, Marie-Fran{\c{c}}oise},
  journal={Discrete \& Computational Geometry},
  volume={39},
  number={4},
  pages={639--655},
  year={2008},
  publisher={Springer}
}

@article{EliNidhiExc:2024,
  title={Network reduction and absence of {H}opf Bifurcations in dual phosphorylation networks with three Intermediates},
  author={Feliu, Elisenda and Kaihnsa, Nidhi},
  journal={Studies in Applied Mathematics},
  volume={157},
  number={3},
  pages={e70268},
  year={2026},
  publisher={Wiley Online Library}
}

@article{Conradispatial,
  title={Conditions for spatial instabilities and pattern formation from monomial steady state parameterizations},
  author={Conradi, Carsten and Mincheva, Maya and Uecker, Hannes},
  journal={arXiv preprint arXiv:2605.16049},
  year={2026}
}

@book{Fei19,
  title={Foundations of Chemical Reaction Network Theory},
  author={Feinberg, Martin},
  year={2019},
  publisher={Springer},
  doi={10.1007/978-3-030-03858-8}
}

@article{CarstenHopfExclusion19,
  title={On the existence of {H}opf bifurcations in the sequential and distributive double phosphorylation cycle},
  author={Conradi, Carsten and Feliu, Elisenda and Mincheva, Maya},
  journal={Mathematical biosciences and engineering},
  volume={17},
  number={1},
  pages={494--513},
  year={2020},
  doi={10.3934/mbe.2020027
}
}

@inproceedings{Dickenstein2019,
  title={Algebra and geometry in the study of enzymatic cascades},
  author={Dickenstein, Alicia},
  booktitle={World Women in Mathematics 2018: Proceedings of the First World Meeting for Women in Mathematics (WM) $^2$},
  pages={57--81},
  year={2019},
  organization={Springer}
}

@article{Dickenstein20,
  title={Algebraic geometry tools in systems biology},
  author={Dickenstein, Alicia},
  journal={Not. Am. Math. Soc.},
  volume={67},
  number={11},
  pages={1706--1715},
  year={2020}
}

@article{conradietal19,
  title={Emergence of oscillations in a mixed-mechanism phosphorylation system},
  author={Conradi, Carsten and Mincheva, Maya and Shiu, Anne},
  journal={Bulletin of Mathematical Biology},
  volume={81},
  number={6},
  pages={1829--1852},
  year={2019},
  publisher={Springer},
  doi={10.1007/s11538-019-00580-6}
}

@article{banaji:2ndaddcomp,
  title={The determinant of the second additive compound of a square matrix: a formula and applications},
  author={Banaji, Murad},
  journal={arXiv preprint arXiv:1806.07162},
  year={2018}
}

@article{Feliu:2020,
  title =        {The kinetic space of multistationarity in dual
                  phosphorylation},
  author =       {Feliu, Elisenda and Kaihnsa, Nidhi and de Wolff,
                  Timo and Y{\"u}r{\"u}k, O{\u{g}}uzhan},
  journal =      {Journal of Dynamics and Differential Equations},
  volume =       {34},
  pages =        {1-28},
  year =         {2020},
  doi =          {10.1007/s10884-020-09889-6}
}

@article{Fiedler85PhD,
  title={An index for global {H}opf bifurcation in parabolic systems.},
  author={Fiedler, Bernold},
  journal={Journal f{\"u}r die reine und angewandte Mathematik},
  volume={358},
  pages={1--36},
  year={1985}
}

@article{gunawardena:2007,
  title={Distributivity and processivity in multisite phosphorylation can be distinguished through steady-state invariants},
  author={Gunawardena, Jeremy},
  journal={Biophysical journal},
  volume={93},
  number={11},
  pages={3828--3834},
  year={2007},
  publisher={Elsevier}
}

@article{cohen2000regulation,
  title={The regulation of protein function by multisite phosphorylation--a 25 year update},
  author={Cohen, Philip},
  journal={Trends in biochemical sciences},
  volume={25},
  number={12},
  pages={596--601},
  year={2000},
  publisher={Elsevier}
}

@article{WangSontagMonotone:2008,
  author  = {Wang, Liming and Sontag, Eduardo D.},
  title   = {Singularly Perturbed Monotone Systems and an Application
             to Double Phosphorylation Cycles},
  journal = {Journal of Nonlinear Science},
  year    = {2008},
  volume  = {18},
  number  = {5},
  pages   = {527--550},
  doi     = {10.1007/s00332-008-9021-2}
}

@article{HolsteinFlockerziConradi:2013,
  author  = {Holstein, Katharina and Flockerzi, Dietrich
             and Conradi, Carsten},
  title   = {Multistationarity in Sequential Distributed Multisite
             Phosphorylation Networks},
  journal = {Bulletin of Mathematical Biology},
  year    = {2013},
  volume  = {75},
  number  = {11},
  pages   = {2028--2058},
  doi     = {10.1007/s11538-013-9878-6}
}

@article{FlockerziHolsteinConradi:2014,
  author  = {Flockerzi, Dietrich and Holstein, Katharina
             and Conradi, Carsten},
  title   = {{$N$}-Site Phosphorylation Systems with {$2N-1$} Steady States},
  journal = {Bulletin of Mathematical Biology},
  year    = {2014},
  volume  = {76},
  number  = {8},
  pages   = {1892--1916},
  doi     = {10.1007/s11538-014-9984-0}
}

@article{ConradiMincheva:2014,
  author  = {Conradi, Carsten and Mincheva, Maya},
  title   = {Catalytic Constants Enable the Emergence of Bistability
             in Dual Phosphorylation},
  journal = {Journal of the Royal Society Interface},
  year    = {2014},
  volume  = {11},
  number  = {95},
  pages   = {20140158},
  doi     = {10.1098/rsif.2014.0158}
}

@article{ErramiEtAl:2015,
  author  = {Errami, Hassan and Eiswirth, Markus and Grigoriev, Dima
             and Seiler, Werner M. and Sturm, Thomas and Weber, Andreas},
  title   = {Detection of {H}opf Bifurcations in Chemical Reaction
             Networks Using Convex Coordinates},
  journal = {Journal of Computational Physics},
  year    = {2015},
  volume  = {291},
  pages   = {279--302},
  doi     = {10.1016/j.jcp.2015.02.050}
}

@article{ConradiFeliuMinchevaWiuf:2017,
  author  = {Conradi, Carsten and Feliu, Elisenda and Mincheva, Maya
             and Wiuf, Carsten},
  title   = {Identifying Parameter Regions for Multistationarity},
  journal = {PLoS Computational Biology},
  year    = {2017},
  volume  = {13},
  number  = {10},
  pages   = {e1005751},
  doi     = {10.1371/journal.pcbi.1005751}
}

@article{BozemanMorales:2017,
  author  = {Bozeman, Luna and Morales, Adriana},
  title   = {No Oscillations in the {Michaelis--Menten} Approximation
             of the Dual Futile Cycle under a Sequential and
             Distributive Mechanism},
  journal = {SIAM Undergraduate Research Online},
  year    = {2017},
  volume  = {10},
  pages   = {21--28},
  doi     = {10.1137/16S015565}
}

@article{Tung:2018,
  author  = {Tung, Hwai-Ray},
  title   = {Precluding Oscillations in {Michaelis--Menten}
             Approximations of Dual-Site Phosphorylation Systems},
  journal = {Mathematical Biosciences},
  year    = {2018},
  volume  = {306},
  pages   = {56--59},
  doi     = {10.1016/j.mbs.2018.10.008}
}

@article{WangSontag:2008,
  author  = {Wang, Liming and Sontag, Eduardo D.},
  title   = {On the Number of Steady States in a Multiple Futile Cycle},
  journal = {Journal of Mathematical Biology},
  year    = {2008},
  volume  = {57},
  number  = {1},
  pages   = {29--52},
  doi     = {10.1007/s00285-007-0145-z}
}

@article{CaiHimmelmannOstermann:2025,
  author  = {Cai, May and Himmelmann, Matthias and Ostermann, Birte},
  title   = {Empirically Exploring the Space of Monostationarity
             in Dual Phosphorylation},
  journal = {Journal of Mathematical Chemistry},
  year    = {2025},
  volume  = {63},
  number  = {3},
  pages   = {666--692},
  doi     = {10.1007/s10910-024-01687-5}
}

@article{householder:1968,
  title={Bigradients and the problem of {R}outh and {H}urwitz},
  author={Householder, Alston S},
  journal={SIAM Review},
  volume={10},
  number={1},
  pages={56--66},
  year={1968},
  publisher={SIAM}
}

@article{Tisseur:01,
  title={The quadratic eigenvalue problem},
  author={Tisseur, Fran{\c{c}}oise and Meerbergen, Karl},
  journal={SIAM review},
  volume={43},
  number={2},
  pages={235--286},
  year={2001},
  publisher={SIAM}
}

\end{document}